\documentclass[a4paper, reqno,12pt]{amsart}

\usepackage{tikz, tikz-3dplot}
\usepackage{amsmath,amsthm,amssymb,graphics}
\usepackage[utf8]{inputenc}
\usepackage[english]{babel}
\usepackage[colorlinks=true, allcolors=blue]{hyperref}

\theoremstyle{plain}
\newtheorem{te}{Theorem}[section]

\newtheorem{pr}[te]{Proposition}

\newtheorem{llll}[te]{Lemma}

\newtheorem*{ack*}{Acknowledgment}
\theoremstyle{remark}

\newtheorem{rmk}{Remark}

\newcommand{\mbd}{\boldsymbol}

\newcommand{\dsum}{\displaystyle\sum}
\newcommand{\dint}{\displaystyle\int}
\newcommand{\doint}{\displaystyle\oint}
\newcommand{\dprod}{\displaystyle\prod}
\numberwithin{equation}{section}

\def\0{{\bf 0}}

\def\R{{\mathbb R}}

\def\N{{\mathbb N}}
\def\C{{\mathbb C}}

\def\Z{{\mathbb Z}}

\begin{document}

\author{Kiseok Yeon}
\email{kyeon@purdue.edu}
\address{Department of Mathematics, Purdue University, 150 N. University Street, West Lafayette, IN 47907-2067, USA}
\title[Major and minor arcs estimation]{Small fractional parts of polynomials and mean values of exponential sums}
\maketitle

\begin{abstract}
   Let $k_i\ (i=1,2,\ldots,t)$ be natural numbers with $k_1>k_2>\cdots>k_t>0$, $k_1\geq 2$ and $t<k_1.$ 
   Given real numbers $\alpha_{ji}\ (1\leq j\leq t,\ 1\leq i\leq s)$, we consider polynomials of the shape 
    $$\varphi_i(x)=\alpha_{1i}x^{k_1}+\alpha_{2i}x^{k_2}+\cdots+\alpha_{ti}x^{k_t},$$ 
   and derive upper bounds for fractional parts of polynomials in the shape $$\varphi_1(x_1)+\varphi_2(x_2)+\cdots+\varphi_s(x_s),$$
   by applying novel mean value estimates related to Vinogradov's mean value theorem. Our results improve on earlier Theorems of Baker (2017).
  % we improve on the earlier results regarding problems of finding an exponent $\sigma>0$, depending on $k,$ such that 
   % $$\min_{\substack{0\leq \boldsymbol{x}\leq X\\ \boldsymbol{x}\neq \boldsymbol{0}}}\|\varphi_1(x_1)+\varphi_2(x_2)+\cdots+\varphi_s(x_s)\|<X^{-\sigma+\epsilon}.$$

    %generalize major and minor arcs estimations introduced by Wooley. Furthermore, we introduce an application of these estimates to small fractional parts of polynomials in additive forms. Our result not only improves earlier theorem of Baker, but also generalizes this theorem.
\end{abstract}

\section{Introduction}

Since the early part of the last century, estimates of Weyl sums have played crucial roles in many problems in additive number theory. The classical bounds for Weyl sums have stemmed from Weyl's method [$\ref{ref17}$] and Vinogradov's method [$\ref{ref18}$]. %The bound from the latter method has been improved by the recent progress of Main conjecture of Vinogradov's mean value theorem.
In particular, 
these bounds have been widely used in studying the distribution of polynomial modulo $1$, initiated by a question posed by Hardy and Littlewood [$\ref{ref19}$] asking, when $\alpha\in \R$, $k\in \N$ and $\epsilon>0$, whether there exists $\sigma>0$ not depening on $\alpha$ such that
\begin{equation*}
    \min_{1\leq x\leq X}\|\alpha x^k\|\leq X^{-\sigma+\epsilon},
\end{equation*}
where $\|\cdot\|$ denotes the distance to the nearest integer and $X$ is sufficiently large in terms of $k$ and $\epsilon.$ By exploiting such bounds for Weyl sums, Heilbronn [$\ref{ref20}$] and Danicic [$\ref{ref21}$] obtained $\sigma=2^{1-k}.$ Subsequently, the exponent $1/2$ in the case $k=2$ was improved to $\sigma=4/7$ by Zaharescu [$\ref{ref29}$]. By exploiting estimates for smooth Weyl sums, Wooley [$\ref{ref11}$] obtained the permissible exponent $\sigma=1/(k\log k+O(k\log\log k)).$ Furthermore, combined with the recent progress on bounds for Weyl sums, stemming from the resolution of the main conjecture in Vinogradov's mean value theorem, Baker [$\ref{ref4}$] shows that $\sigma=1/(k(k-1))$ is permissible, and also derives the explicit exponent $\sigma(s,k)=s/(k(k-1))$ such that 
\begin{equation}\label{1}
    \min_{\substack{0\leq \boldsymbol{x}\leq X\\\boldsymbol{x}\neq \boldsymbol{0}}}\|\alpha_1x^k_1+\cdots+\alpha_s x_s^k\|\leq X^{-\sigma(s,k)+\epsilon},
\end{equation}
for $1\leq s\leq k(k-1)$. Here and throughout, we write $0\leq \boldsymbol{x}\leq X$ and $\boldsymbol{x}\neq \boldsymbol{0}$ to abbreviate the conditions $0\leq x_1,\ldots,x_s\leq X$ and $(x_1,\ldots,x_s)\neq (0,\ldots,0).$

In this paper, we seek to make the bound $(\ref{1})$ sharper via mean values of exponential sum, rather than exploiting bounds for Weyl sums. Furthermore, by applying new mean value estimates for exponential sums related to Vinogradov's mean value theorem, the method described here shall deliver bounds for small fractional parts of polynomial in the generalized shape $\varphi_1(x_1)+\varphi_2(x_2)+\cdots+\varphi_s(x_s),$ where $$\varphi_i(x)=\alpha_{1i}x^{k_1}+\alpha_{2i}x^{k_2}+\cdots+\alpha_{ti}x^{k_t}$$
in which $s,t,k_1,k_2,\ldots,k_t$ are natural numbers with $k_1>t\geq 2$ and $k_1>k_2>\cdots>k_t.$

\bigskip

\begin{te}
Let $\epsilon>0$ and $s,k$ be natural numbers with $k\geq 6.$ Suppose that $X$ is sufficiently large in terms of $s,k$ and $\epsilon.$ Consider $\alpha_i\in \R$ with $1\leq i\leq s.$ Then, whenever $s\geq \frac{k(k+1)}{2}$, one has
\begin{equation}
    \min_{\substack{0\leq \boldsymbol{x}\leq X\\ \boldsymbol{x}\neq \boldsymbol{0}}}\|\alpha_1x_1^k+\alpha_2x_2^k+\cdots+\alpha_sx_s^k\| \leq X^{-1+\epsilon}.
\end{equation}

\end{te}

For comparison, the work of Baker [$\ref{ref4}$, Theorem 3] shows ($\ref{1}$) with $\sigma(s,k)=\frac{s}{k(k-1)}$ for $1\leq s\leq k(k-1).$ %that
 %$$\min_{\substack{0\leq \boldsymbol{x}\leq X\\ \boldsymbol{x}\neq \boldsymbol{0}}}\|\alpha_1x_1^k+\alpha_2x_2^k+\cdots+\alpha_sx_s^k\|\leq X^{-\frac{s}{k(k-1)}+\epsilon}$$ for $1\leq s\leq k(k-1)$. 
  His work also gives results when $s> k(k-1),$ too complicated to state in full here. It is sufficient to report that the exponent $s/(k(k-1))$ is replaced by an exponent $\sigma$ in Baker [$\ref{ref4}$, Theorem 3], with $\sigma\rightarrow 2$ as $s\rightarrow \infty.$
 Theorem 1.1 improves on this result when $\frac{k(k+1)}{2}\leq s< k(k-1)$. 
 
 We note that with additional effort, for $s\geq k+2$ one may get ($\ref{1}$) with
% $$\min_{\substack{0\leq \boldsymbol{x}\leq X\\ \boldsymbol{x}\neq \boldsymbol{0}}}\|\alpha_1x_1^k+\alpha_2x_2^k+\cdots+\alpha_sx_s^k\|\leq X^{-\sigma(s,k)+\epsilon},$$
% where 
 \begin{equation}\label{333}
 \sigma(s,k)=\min\biggl\{\frac{s}{k(k+1)-s}, 1\biggr\}.
 \end{equation}
% However, since the technical complications required to obtain ($\ref{333}$) may obscure key ideas of our arguments, we will not provide the proof of ($\ref{333}$) in this paper. 
Notice that this improves on a result of Baker [$\ref{ref4}$] described above when $2k<s<k(k-1)$. We record this result in section 4 (see Theorem 4.1 below). We also note that experts may expect that the exponent $(\ref{333})$ can be improved for large $k$ by using estimates for smooth Weyl sums. However, to obtain results for $s>1$ one encounters a number of technical complications that threaten to obstruct useful conclusions. Consequently, we focus in this paper on conclusions made accessible by our new mean value estimates for exponential sums.

\bigskip
 As we explained above, the method described here delivers bounds for small fractional parts of more general polynomials. Thus, in order to describe these polynomials and the following theorems, we require some notation. Consider a fixed $t$-tuple $\mathbf{k}=(k_1,\ldots,k_t)$ of positive integers satisfying 
\begin{equation*}
   k=k_1>k_2>\cdots>k_t\geq 1.
\end{equation*}
We denote $\{1,2,\ldots,k_1\}\setminus\{k_1,\ldots,k_t\}$ by $\{i_1,\ldots,i_{k-t}\}$ with $i_1>\cdots>i_{k-t}.$ Furthermore, we write $\sigma=\sigma(\mathbf{k})$ for
%Define an exponential sum $f(\alpha_{1},\mbd{\alpha})=f(\alpha_{1},\mbd{\alpha};X)$ by
%$$f(\alpha_{1},\mbd{\alpha})=\dsum_{1\leq x\leq X}e(\alpha_1x^k+\alpha_{2}x^{k-1}+\cdots+\alpha_kx)$$
\begin{align}\label{eq1.41.4}
\sigma= \max_{1\leq l\leq k-t} \frac{l}{(k-i_l)(k-i_l+1)}.
\end{align}

\bigskip

\begin{te}
Let $\epsilon>0$. Suppose that $s,t, k_1,\ldots,k_t$ are natural numbers satisfying $k_1\geq 6$, $k_1>t\geq 2$ and $k_1>k_2>\cdots>k_t.$ Suppose that $X$ is sufficiently large in terms of $s,k_1$ and $\epsilon.$ Consider $\alpha_{ji}\in \R$ with $1\leq i\leq s$ and $1\leq j\leq t$. Define $\varphi_i(x)=\alpha_{1i}x^{k_1}+\cdots+\alpha_{ti} x^{k_t}$ with $1\leq i\leq s$. Then, whenever $s>k_1^2+k_1+2\lceil\sigma(1-k_1)\rceil$, one has

\begin{equation}\label{4}
    \min_{\substack{0\leq \boldsymbol{x}\leq X\\\boldsymbol{x}\neq \boldsymbol{0}}}\|\varphi_1(x_1)+\varphi_2(x_2)+\cdots +\varphi_{s}(x_{s})\|\leq X^{-1+\epsilon}.
\end{equation}
\end{te}

%This may seem not as strong as Theorem 1.1, in a aspect of the range of $s.$ However, we note that from [$\ref{ref2}$, Theorem 4.4] together with Vinogradov's mean value theorem that, in general, for every $\epsilon>0$ one has
%\begin{equation}\label{555}
 % \sup_{\boldsymbol{\alpha}\in \mathfrak{m}} \bigl| \dsum_{1\leq x\leq X}e(\alpha_{k_1}x^{k_1}+\alpha_{k_2}x^{k_2}+\cdots+\alpha_{k_t}x^{k_t})\bigr|\ll X^{1-1/(2k_1(k_1-1))+\epsilon},
%\end{equation}
% where $\mathfrak{m}$ is defined to be $[0,1)^t\setminus \mathfrak{M}$ in which
% \begin{equation*}
 %    \mathfrak{M}=\bigcup_{\substack{0\leq a_1,\ldots,a_t\leq q\leq X\\(q,a_1,\ldots,a_t)=1}}\mathfrak{M}(q,a_1,\ldots,a_t)
 %\end{equation*}
%and
%\begin{equation*}
 %   \mathfrak{M}(q,a_1,\ldots,a_t)=\{\boldsymbol{\alpha}\in [0,1)^t|\  |q\alpha_{k_j}-a_{k_j}|\leq X^{1-k_j}\}.
%\end{equation*}
%Thus, one might expect from ($\ref{555}$) that the result of Baker [$\ref{ref3}$] suggests that whenever $s\geq2k_1(k_1-1),$ one obtains ($\ref{4}$). Therefore, we see that Theorem 1.2 improves the exponent $2k_1(k_1-1)$ by a factor nearly as large as $1/2.$
The reader will observe that the condition on $s$ in the conclusion of Theorem 1.2 is almost twice as restrictive as that in Theorem 1.1. The explanation for this reduction in strength lies with the generality of the polynomials $\varphi_i$, and correspondingly weaker estimate available for associated exponential sums.

\bigskip

To describe the following theorems regarding new mean values of exponential sums, we introduce some notation.
Define the exponential sum
$F(\alpha_{k_1},\mbd{\alpha}^{t-1})=F_{\mathbf{k}}(\alpha_{k_1},\ldots,\alpha_{k_t};X)$ by 
$$F(\alpha_{k_1},\mbd{\alpha}^{t-1})=\dsum_{1\leq x \leq X}e(\alpha_{k_1}x^{k_1}+\alpha_{k_2}x^{k_2}+\cdots+\alpha_{k_t}x^{k_t}).$$
Denote $d\alpha_{k_t}d\alpha_{k_{t-1}}\cdots d\alpha_{k_2}$ by $d\mbd{\alpha}^{t-1}$, and write
$$\doint  |F(\alpha_{k_1},\mbd{\alpha}^{t-1})|^{2s}d\mbd{\alpha}^{t-1}=\dint_{[0,1)^{t-1}}  |F(\alpha_{k_1},\mbd{\alpha}^{t-1})|^{2s}d\alpha_{k_t}d\alpha_{k_{t-1}}\cdots d\alpha_{k_2}.$$ Furthermore, we write
$$f(\alpha_{k_1},\boldsymbol{\alpha})=\dsum_{1\leq x\leq X}e(\alpha_{k_1}x^{k_1}+\alpha_{k_1-1}x^{k_1-1}+\cdots+\alpha_1 x)$$
and
$$ \doint|f(\alpha_{k_1},\boldsymbol{\alpha})|^{2s}d\boldsymbol{\alpha}=\dint_{[0,1)^{k_1-1}}|f(\alpha_{k_1},\boldsymbol{\alpha})|^{2s} d\boldsymbol{\alpha}.$$
%Especially, when $\mathbf{k}=(k,k-1,...,2,1)$, we write $f(\alpha_k,\mbd{\alpha})$ for $F(\alpha_k,\mbd{\alpha}^{k-1}).$

\begin{te}
%Suppose that  $k(=k_1),k_2,...,k_t$ are distinct positive integers with $k_1>k_2>\cdots>k_t\geq 1$. Suppose that $\{i_1, \cdots, i_{k-t}\}$ are the complement set of $\{k_1,k_2,...,k_t\}$ with $k_1> i_1>i_2>\cdots>i_{k-t}\geq 1$, $\{1,...,k_1\}\setminus\{k_t,...,k_1\}=\{i_{k-t},...,i_1\}$. 
 Let $s,t$ and $k$ be natural numbers with $t<k$. Let $l$ be an integer with $1\leq l\leq k-t$. Consider a rational approximation to $\alpha_{k}$ satisfying $|\alpha_{k}-a/q|\leq 1/q^2$ with $(q,a)=1$. Then, for  $\epsilon>0$, one has
\begin{equation*}
\doint\left|F(\alpha_k,\mbd{\alpha}^{t-1})\right|^{2s}d\mbd{\alpha}^{t-1}\ll R_lX^{i_1+\cdots+i_{k-t}+\epsilon}\doint\left|f(\alpha_k,\mbd{\alpha})\right|^{2s}d\mbd{\alpha},    
\end{equation*}
where $$R_l=\dprod_{j=1}^{l}\left(X^{-i_j}+X^{-k+i_j}+q^{-1}+qX^{-k}\right)^{\frac{1}{(k-i_l)(k-i_l+1)}}.$$
\end{te}

As a consequence of Theorem 1.3, one finds that the mean value over all coefficients but the leading coefficient has an upper bound in terms of the denominator of the rational approximation to $\alpha_k$. From this, we obtain mean value estimates by integrating over $\alpha_k$ lying over major arcs and minor arcs, respectively.

In order to describe these estimates, which we record in Theorem 1.4, and for the argument used throughout this paper, we must introduce sets of major arcs and minor arcs.
Define the major arcs $\mathfrak{M}_l$ with $l>0$ by
\begin{equation}\label{eq1.6}
    \mathfrak{M}_l=\bigcup_{\substack{0\leq a\leq q \leq X\\(q,a)=1}}\mathfrak{M}_l(q,a),
\end{equation}
where $\mathfrak{M}_l(q,a)=\{\alpha\in [0,1)|\ |q\alpha-a|\leq (lk)^{-1}X^{-k+1}\}$. Define the minor arcs to be $\mathfrak{m}_l=[0,1)\setminus \mathfrak{M}.$ We abbreviate $\mathfrak{M}_2$ simply to $\mathfrak{M}.$ Throughout this paper, we use $\mathfrak{M}$ and $\mathfrak{m}$ without further comments, unless specified otherwise. Furthermore, we recall the definition $(\ref{eq1.41.4})$ of the exponent $\sigma,$ 
and write $D$ for 
\begin{equation}\label{1.7}
D=k_1+k_2+\cdots+k_t.    
\end{equation}
%\begin{equation*}
%\sigma= \max_{1\leq l\leq k-t} \frac{l}{(k-i_l)(k-i_l+1)}.    
%\end{equation*}

\begin{te}\label{thm1.2}
One has the following:

$(\romannumeral1)$
When $s$ is a natural number with $2s\geq k^2+(1-2\sigma)k+2\sigma$, one has
\begin{equation}\label{7777}
    \dint_{\mathfrak{M}}\doint\left|F(\alpha_{k},\mbd{\alpha}^{t-1})\right|^{2s}d\mbd{\alpha}^{t-1}d\alpha_k\ll X^{2s-D+\epsilon}.
\end{equation}

$(ii)$
When $s$ is a natural number with $2s\geq k(k+1),$ one has
\begin{equation}\label{8888}
    \dint_{\mathfrak{m}}\doint\left|F(\alpha_{k},\mbd{\alpha}^{t-1})\right|^{2s}d\mbd{\alpha}^{t-1}d\alpha_k\ll X^{2s-D-\sigma+\epsilon}.
\end{equation}

\end{te}

%By exploiting $(\ref{8888})$ with $F(\alpha_k,\boldsymbol{\alpha}^{t-1})=\dsum_{1\leq x\leq X}e(\alpha_kx^k),$ Wooley [] establishes improvements in the number of variables required to establish the asymptotic formula in Waring's problem. , by exploiting analogous 
Wooley [$\ref{ref12}$, Theorem 1.3] provided the mean value estimates of exponential sums over minor arcs, which is $(\ref{8888})$ with $F(\alpha_k,\boldsymbol{\alpha})=\sum_{1\leq x\leq X} e(\alpha_k x^k).$ This mean value estimate delivered improvements in the number of variables required to establish the asymptotic formula in Waring's problem, the density of integral solutions of diagonal Diophantine equations and slim exceptional sets for the asymptotic formula in Waring's problem. Wooley [$\ref{ref13}$, Theorem 1.1] established an essentially optimal estimate for ninth moment of exponential sum having argument $\alpha x^3+\beta x$ (see also [$\ref{ref28}$, Theorem 1.3]), by introducing ($\ref{8888}$) with $F(\alpha_3,\boldsymbol{\alpha})=\sum_{1\leq x\leq X} e(\alpha_3x^3+\alpha_1x).$ Furthermore, Wooley [$\ref{ref15}$, Theorem 14.4] recorded bounds for $(\ref{7777})$ and $(\ref{8888})$ with $k_2<k_1-1.$ In Theorem 1.4, we provide mean values of $F(\alpha_{k},\mbd{\alpha}^{t-1})=\sum_{1\leq x \leq X}e(\alpha_{k}x^{k_1}+\alpha_{k_2}x^{k_2}+\cdots+\alpha_{k_t}x^{k_t})$ with no restrictions on the exponents $k_1,\ldots,k_t.$ Combining with Theorem 1.4, the method described in the proof of Theorem 1.1 shall deliver the proof of Theorem 1.2. 

We also note that by applying H\"older's inequality and the trivial bound $\left|F(\alpha_{k},\mbd{\alpha}^{t-1})\right|\leq X$,
 it follows from Theorem 1.4 $(\romannumeral2)$ %together with Vinogradov's mean value theorem 
 that there exists $s_0$ with $s_0<k(k+1)/2$ such that whenever $s\geq s_0$ we have $$\dint_{\mathfrak{m}}\doint\left|F(\alpha_{k},\mbd{\alpha}^{t-1})\right|^{2s}d\mbd{\alpha}^{t-1}d\alpha_k\ll X^{2s-D+\epsilon}.$$
Therefore, we find that there exists $s_0$ with $s_0<\frac{k(k+1)}{2}$ such that whenever $s\geq s_0$ one has
\begin{equation*}
\begin{aligned}
    &\dint\doint\left|F(\alpha_{k},\mbd{\alpha}^{t-1})\right|^{2s}d\mbd{\alpha}^{t-1}d\alpha_k\\
    &=\dint_{\mathfrak{M}}\doint\left|F(\alpha_{k},\mbd{\alpha}^{t-1})\right|^{2s}d\mbd{\alpha}^{t-1}d\alpha_k+\dint_{\mathfrak{m}}\doint\left|F(\alpha_{k},\mbd{\alpha}^{t-1})\right|^{2s}d\mbd{\alpha}^{t-1}d\alpha_k\ll X^{2s-D+\epsilon}.
\end{aligned}
\end{equation*}
%Especially, as in [], combining these estimations ($\sigma=1$) and [, Lemma 5.2]
%\begin{equation*}
 %   \dint\dint|\dsum_{1\leq x\leq X}e(\alpha x^3+\beta x)|^6d\alpha d\beta\ll X^{3+\epsilon},
%\end{equation*}
%it follows from Holder's inequality that
%\begin{align*}
 %  & \dint\dint|\dsum_{1\leq x\leq X}e(\alpha x^3+\beta x)|^{10}d\beta d\alpha\\
 %  &\ll \dint_{\mathfrak{M}}\dint|\dsum_{1\leq x\leq X}e(\alpha x^3+\beta x)|^{10}d\beta d\alpha+\dint_{\mathfrak{M}}\dint|\dsum_{1\leq x\leq X}e(\alpha x^3+\beta x)|^{10}d\beta d\alpha\ll X^{6+\epsilon}
%\end{align*}
This range of $s$ is superior to those trivially obtained by Vinogradov's mean value theorem.

%We remark that for fixed $r$, $\sigma_0(\floor{\frac{k}{2}})=\frac{1}{2}-o(k).$ In other words, if $k$ is sufficiently large in terms of $r$, then one can choose $r_0$ with $r\leq r_0 \leq \floor{\frac{k}{2}}$ so that $\sigma_0(r_0)$ is nearly $\frac{1}{2}.$

 %With a variant of the proof of Theorem 3, one can also deduce the case $s\leq \frac{k(k+1)}{2}$. However, we deal with only the case $s> \frac{k(k+1)}{2}$. This is because that notational complexity may prevent from giving the idea in how to apply major and minor arcs estimation to small fraction parts of polynomials.
 
 \bigskip
 
 %Theorem 1.4 is concerned with small fraction parts of $\varphi_1(n_1)+\varphi_2(n_2)+\cdots+\varphi_{s}(n_{s})$ where $\varphi_i(n_i)=\alpha_{1i}n_i^{k_1}+\cdots+\alpha_{ti}n_i^{k_t}$ with $\alpha_{ij}\in\R$ and $k_j\in\N$.
 The consequences of Theorem 1.2 and Theorem 1.4 are dependent on $\sigma$, which is the quantity determined by $\mathbf{k}=(k_1,\ldots,k_t)$. Thus, we shall see how this quantity $\sigma$ varies according to the number of exponents and its arrangement.

 Recall the definition ($\ref{eq1.41.4}$) of the exponent $\sigma$ and that $\{i_1,\ldots,i_{k-t}\}=\{1,2\ldots,k_1\}\setminus\{k_t,\ldots,k_1\}$ with $i_1>\cdots>i_{k-t}.$ Then, we observe following:

\ 

 (1) Let $\mathbf{k}=(k,k-1,\ldots,k-(t-1))$ with $t<k/2.$ Then, by taking $l=t,$ one obtains
 
 $$\sigma=\max_{1\leq l\leq k-t}\frac{l}{(k-i_l)(k-i_l+1)}\geq\frac{t}{(2t-1)(2t)}=O(t^{-1}).$$  %$\varphi_i(n)=\alpha_{1i}n^k+\alpha_{2i}n^{k-1}+\cdots+\alpha_{ti}n^{k-(t-1)}$. This is the case that $(k_1,...,k_t)=(k,k-1,...,k-(t-1))$. 
 
 \bigskip
 
 (2) Let $t=m_1+m_2$. Let $$\mathbf{k}=(k,k-1,\ldots,k-(m_1-1),m_2,\ldots,1)$$
 with $m_1+m_2<k/2.$
 %Let $\varphi_{i}(n)= \alpha_1n^k+\alpha_2n^{k-1}+\cdots+\alpha_{m_1}n^{k-(m_1-1)}+\cdots+\alpha_{m_2}n^{m_2}+\cdots+\alpha_tn$. This is the case that $(k_1,...,k_t)=(k,k-1,...,k-(m_1-1), m_2,...,1)$.
 Then, by taking $l=m_1,$ one has $$\sigma=\max_{1\leq l\leq k-t}\frac{l}{(k-i_l)(k-i_l+1)}\geq\frac{m_1}{(2m_1-1)(2m_1)}=O(m_1^{-1}).$$
 
 \bigskip
 
 (3) Let $\mathbf{k}=(k,k_2,\ldots,k_t)$ with $k_1=k, k_2=k-1$ and $k_3\neq k-2.$ Then, by taking $l=1$, one has
 $\sigma=1/2.$

 \bigskip
 
Thus, if we assume that $\mathbf{k}=(k_1,\ldots,k_t)$ with $t<k/2,$ then one infers from the observations above that $\sigma$ is at least $O(t^{-1})$.

In section 2, we provide the proof of Theorem 1.3. The method of the proof of Theorem 1.3 mainly follows the argument in [$\ref{ref12}$] together with the argument used in [$\ref{ref5}$]. In section 3, we provide the proof of Theorem 1.4, by making use of Theorem 1.3. In section 4, we introduce applications of mean values of exponential sums to fractional parts of polynomials and provide the proof of Theorem 1.1. Furthermore, we record in Theorem 4.1 a more quantitative result than Theorem 1.1 and provide its proof at the end of section 4. In section 5, we give the proof of Theorem 1.2 by exploiting Theorem 1.4 and the method introduced in section 4.
Throughout this paper, we use $\gg$ and $\ll$ to denote Vinogradov's well-known notation, and write $e(z)$ for $e^{2\pi iz}$. We adopt the convention that whenever $\epsilon$ appears in a statement, then the statement holds for each $\epsilon>0$, with implicit constants depending on $\epsilon.$
\section*{Acknowledgment}
The author acknowledges support from NSF grant DMS-2001549 under the supervision of Trevor Wooley. The author is grateful for support from Purdue University. Especially, the author would like to thank Trevor Wooley for careful reading and helpful comments which have improved the exposition.

\bigskip
\section{Proof of Theorem 1.3}

In this section, we provide three lemmas and combine all to prove Theorem 1.3. 
\subsection{Auxiliary lemmas}

In order to describe Lemma 2.1, we recall that 
\begin{align*}
    F(\alpha_{k_1},\mbd{\alpha}^{t-1})=\dsum_{1\leq x \leq X}e(\alpha_{k_1}x^{k_1}+\alpha_{k_2}x^{k_2}+\cdots+\alpha_{k_t}x^{k_t})
\end{align*}
and
\begin{equation*}
    f(\alpha_{k_1},\boldsymbol{\alpha})=\dsum_{1\leq x\leq X}e(\alpha_{k_1}x^{k_1}+\alpha_{k_1-1}x^{k_1-1}+\cdots+\alpha_1 x).
\end{equation*}
Furthermore, recall 
$\{i_1,\ldots,i_{k-t}\}=\{1,2,\ldots,k_1\}\setminus\{k_1,\ldots,k_t\}.$ In advance of the statement of the following lemma, we define $\mathcal{I}(\alpha_k):=\mathcal{I}(\alpha_k;l)$ with $1\leq l\leq k-t$ by
$$\mathcal{I}(\alpha_k)=\dsum_{|g_{i_1}|\leq sX^{i_1}}\cdots\dsum_{|g_{i_l}|\leq sX^{i_l}}\doint|f(\alpha_{k},\mbd{\alpha})|^{2s} e(-\mbd{\alpha}^{(l)}\cdot\mbd{g})d\mbd{\alpha},$$
where $d\mbd{\alpha}$=$d\alpha_{k-1}\cdots d\alpha_{1}$ and $\mbd{\alpha}^{(l)}\cdot\mbd{g}=\alpha_{i_1}g_{i_1}+\cdots+\alpha_{i_l}g_{i_l}.$

\begin{llll}
For any $l$ with $1\leq l\leq k-t,$ we have
\begin{align*}
    \doint|F(\alpha_{k},\mbd{\alpha}^{t-1})|^{2s}d\mbd{\alpha}^{t-1} \ll X^{i_{l+1}+i_{l+2}+\cdots+i_{k-t}}\mathcal{I}(\alpha_k).
\end{align*}
\end{llll}

\begin{proof}

Denote by $F(\alpha_{k},\mbd{\alpha}^{t-1},\mbd{\beta}^{k-t-l})=F(\alpha_{k_1},\ldots,\alpha_{k_t},\beta_{l+1},\ldots,\beta_{k-t};X)$ the exponential sum
$$\dsum_{1\leq x \leq X}e(\alpha_{k_1}x^{k_1}+\alpha_{k_2}x^{k_2}+\cdots+\alpha_{k_t}x^{k_t}+\beta_{l+1}x^{i_{l+1}}+\cdots+\beta_{k-t}x^{i_{k-t}}).$$
Furthermore, we denote
\begin{equation*}
    \sigma_{s,j}(\mathbf{x})=\dsum_{i=1}^{s}(x_i^j-x_{s+i}^j)\ \ \ \ \ \  (1\leq j\leq k)
\end{equation*}
and recall $k=k_1.$ We emphasize that in order to suppress multiple layer of suffices, it is convenient to write $k$ in place of $k_1$ in many places. 

As a preliminary manoeuvre, we represent the mean value involving $F(\alpha_{k},\mbd{\alpha}^{t-1})$ in terms of an analogous one involving $F(\alpha_{k},\mbd{\alpha}^{t-1},\mbd{\beta}^{k-t-l}).$ Observe that when $\mathbf{m}=(m_{l+1},\ldots, m_{k-t})\in \Z^{k-t-l},$ if we define
 \begin{equation*}
   G(\alpha_k, \mathbf{m}):=\doint|F(\alpha_{k},\mbd{\alpha}^{t-1},\mbd{\beta}^{k-t-l})|^{2s}e(-\beta_{l+1}m_{l+1}-\cdots-\beta_{k-t}m_{k-t})d\mbd{\beta}^{k-t-l}d\mbd{\alpha}^{t-1},
 \end{equation*}
 then one has
   \begin{align}\label{99999}
    \begin{aligned}
   G(\alpha_k, \mathbf{m}) =\dsum_{1\leq\mathbf{x}\leq X}\delta(\mathbf{x}, \mathbf{m})\doint e(\alpha_{k_1}\sigma_{s,k_1}(\mathbf{x})+\cdots+\alpha_{k_t}\sigma_{s,k_t}(\mathbf{x}))d\mbd{\alpha}^{t-1},
    \end{aligned}
\end{align}
where
\begin{equation*}
    \delta(\mathbf{x}, \mathbf{m})=\dprod_{j=l+1}^{k-t}\left(\dint_0^1e(\beta_{j}(\sigma_{s,i_j}(\mathbf{x})-m_{j}))d\beta_{i_j} \right).
\end{equation*}
By orthogonality, one has
\begin{equation*}
       \dint_0^1e(\beta_{j}(\sigma_{s,i_j}(\mathbf{x})-m_{j}))d\beta_{j}=\left\{\begin{array}{l} 1,\ \ \textrm{when}\ \sigma_{s,i_j}(\mathbf{x})=m_{j},\\
       0,\ \ \textrm{when}\ \sigma_{s,i_j}(\mathbf{x})\neq m_{j}.\end{array}\right.
\end{equation*}
When $1\leq\mathbf{x}\leq X,$ moreover, one has $|\sigma_{s,i_j}(\mathbf{x})|\leq sX^{i_j}\ (l+1\leq j\leq k-t),$ and so 

$$\dsum_{|m_{l+1}|\leq sX^{i_{l+1}}}\cdots\dsum_{|m_{k-t}|\leq sX^{i_{k-t}}}\delta(\mathbf{x}, \mathbf{m})=1.$$
Consequently, on noting that 
$$\dsum_{1\leq \mathbf{x}\leq X}e(\alpha_{k}\sigma_{s,k_1}(\mathbf{x})+\alpha_{k_2}\sigma_{s,k_1}(\mathbf{x})+\cdots+\alpha_{k_t}\sigma_{s,k_t}(\mathbf{x}))=|F(\alpha_{k},\mbd{\alpha}^{t-1})|^{2s},$$
 we deduce from ($\ref{99999}$) that 
 \begin{equation}\label{9999}
 \begin{aligned}
&\dsum_{|m_{l+1}|\leq sX^{i_{l+1}}}\dsum_{|m_{k-t}|\leq sX^{i_{k-t}}}G(\alpha_k,\mathbf{m})\\
&=\doint\dsum_{1\leq\mathbf{x}\leq X}\biggl(\dsum_{\mathbf{m}}\delta(\mathbf{x},\mathbf{m})\biggr)e(\alpha_{k}\sigma_{s,k_1}(\mathbf{x})+\alpha_{k_2}\sigma_{s,k_1}(\mathbf{x})+\cdots+\alpha_{k_t}\sigma_{s,k_t}(\mathbf{x}))d\mbd{\alpha}^{t-1}\\
&=\doint|F(\alpha_{k},\mbd{\alpha}^{t-1})|^{2s}d\mbd{\alpha}^{t-1}.
\end{aligned}
 \end{equation}
   Therefore, it follows from ($\ref{99999}$) and ($\ref{9999}$) with the triangle inequality that 
   \begin{equation}\label{5'}
   \begin{aligned}
 & \doint|F(\alpha_{k},\mbd{\alpha}^{t-1})|^{2s}d\mbd{\alpha}^{t-1}\\
  &\leq \dsum_{|m_{l+1}|\leq sX^{i_{l+1}}}\cdots\dsum_{|m_{k-t}|\leq sX^{i_{k-t}}} \doint|F(\alpha_{1},\mbd{\alpha}^{t-1},\mbd{\beta}^{k-t-l})|^{2s}d\mbd{\beta}^{k-t-l}d\mbd{\alpha}^{t-1}  \\
  &\ll X^{i_{l+1}+i_{l+2}+\cdots+i_{k-t}}\doint|F(\alpha_{1},\mbd{\alpha}^{t-1},\mbd{\beta}^{k-t-l})|^{2s} d\mbd{\beta}^{k-t-l}d\mbd{\alpha}^{t-1}. 
   \end{aligned}
   \end{equation}
   
Next, an argument similar to that used above allows us to show that 
\begin{equation}\label{5}
\begin{aligned}
    &\doint|F(\alpha_{k},\mbd{\alpha}^{t-1},\mbd{\beta}^{k-t-l})|^{2s} d\mbd{\beta}^{k-t-l}d\mbd{\alpha}^{t-1}\\
    &=\dsum_{|g_{i_1}|\leq sX^{i_1}}\cdots\dsum_{|g_{i_l}|\leq sX^{i_l}}\doint|f(\alpha_{k},\mbd{\alpha})|^{2s} e(-\mbd{\alpha}^{(l)}\cdot\mbd{g})d\mbd{\alpha}.
\end{aligned}
\end{equation}
 Thus, on substituting ($\ref{5}$) into ($\ref{5'}$), we complete the proof of Lemma 2.1. 
\end{proof}

%In comparison with the argument in [$\ref{ref12}$], we use $g_{i_{1}},\ldots,g_{i_{l}}$ rather than use a single $g.$ Furthermore, we shall obtain savings for all $g_{i_1},\ldots,g_{i_l}.$ 

%In the latter part of proof, we shall obtain the saving for each summation over $h_{i_j}$ with $1\leq j\leq l$. Thus, one can choose suitable $l$ to maximize such savings in the end.

\bigskip

\bigskip

%The next step is to represent the right hand side in ($\ref{5}$) as a different mean value type, which enables us to obtain savings from all summations over $g_{i_1},\ldots,g_{i_l}.$

In order to describe Lemma 2.2, we require a preliminary step. Observe that by shifting the variable of summation, for each integer $y$ one has
\begin{equation}\label{6}
f(\alpha_{k},\mbd{\alpha})=\dsum_{1+y\leq x\leq X+y}e(\psi(x-y;\alpha_k,\mbd{\alpha})),
\end{equation}
where 
$$\psi(z;\alpha_k,\mbd\alpha)=\alpha_1z+\cdots+\alpha_kz^k.$$
But as a consequence of the Binomial Theorem, if we adopt the convention that $\alpha_0=0$, then we may write $\psi(x-y;\alpha_k,\mbd\alpha)$ in the shape

$$\psi(x-y;\alpha_k,\mbd{\alpha})=\dsum_{i=0}^k\beta_ix^i,$$
where
$$\beta_i=\dsum_{j=i}^k \binom{j}{i}(-y)^{j-i}\alpha_j\ \  (0\leq i \leq k).$$
Write 
\begin{equation}\label{2.62.6}
K(\gamma)=\dsum_{1\leq z \leq X}e(-\gamma z).   
\end{equation}
Then we deduce from ($\ref{6}$) that when $1\leq y\leq X$, one has
\begin{equation}\label{7}
f(\alpha_k,\mbd{\alpha})=\dint_0^1f_y(\alpha_k,\mbd{\alpha};\gamma)K(\gamma)d\gamma,
\end{equation}
where we have written 
$$f_y(\alpha_k,\mbd{\alpha}; \gamma)=\dsum_{1\leq x\leq 2X}e(\psi(x-y;\alpha_k,\mbd{\alpha})+\gamma(x-y)).$$

Define
\begin{equation*}\label{2.72.7}
\mathcal{F}_y(\alpha_k,\mbd{\alpha};\mbd{\gamma})=\dprod_{i=1}^{s}f_y(\alpha_k,\mbd{\alpha};\gamma_i)f_y(-\alpha_k,-\mbd{\alpha};-\gamma_{s+i}),    
\end{equation*}
and 
$$\omega_{y,\boldsymbol{\gamma}}=e(-(\gamma_1+\cdots+\gamma_s-\gamma_{s+1}-\cdots-\gamma_{2s})y)=e(-\Gamma y).$$

%\begin{equation}
 %  \omega_{y,\boldsymbol{\gamma}}=e(-(\gamma_1+\cdots+\gamma_s-\gamma_{s+1}-\cdots-\gamma_{2s})y)= e(-\Gamma y).
%\end{equation}
To facilitate the statement of Lemma 2.2, it is convenient to introduce some notation. Recall $\{i_1,\ldots,i_{k-t}\}=\{1,2,\ldots,k_1\}\setminus\{k_1,\ldots,k_t\}.$ Furthermore, we adopt the notation $\alpha_i=0$ for $i\notin \{1,\ldots, k\}.$ 
%It is worth noting that the summation condition on the right hand side in  $(\ref{2.82.8})$ ensures that no contribution arises when $l> k-m.$ 
Then, we define the exponential sum $\Xi(\alpha_k,\mbd{\alpha})=\Xi(\alpha_k,\mbd{\alpha};l;\boldsymbol{\gamma})$ with $1\leq l\leq k-t$ by
\begin{equation*}\label{2.112.11}
     \Xi(\alpha_k,\mbd{\alpha}) =X^{-1}\dsum_{1\leq y\leq X}\dsum_{|h_{i_1}|\leq sX^{i_1}}\cdots\dsum_{|h_{i_l}|\leq sX^{i_l}}\omega_{y,\mbd{\gamma}}e\biggl(-\dsum_{ m=0}^{k-i_l}\delta_{m} y^{m}\biggr),
\end{equation*}
where 
\begin{equation}\label{2.82.8}
    \delta_{m}=\dsum_{n=1}^l\alpha_{m+i_n}\binom{m+i_n}{i_n}h_{i_n}.
\end{equation}
Therefore, on recalling that the definition of $\mathcal{I}(\alpha_k):=\mathcal{I}(\alpha_k;l)$ in the statement of Lemma 2.1, we have the following lemma.
\begin{llll}
For any $l$ with $1\leq l\leq k-t,$ we have
\begin{equation*}
    \mathcal{I}(\alpha_k)\ll \doint \doint \mathcal{F}_0(\alpha_k,\mbd{\alpha};\mbd{\gamma})\Xi(\alpha_k,\mbd{\alpha})\tilde{K}(\mbd{\gamma})d\mbd{\alpha}d\mbd{\gamma},
\end{equation*}
where $\tilde{K}(\mbd{\gamma})=\dprod_{i=1}^s K(\gamma_i)K(-\gamma_{s+i}).$
\end{llll}
\begin{proof}

On substituting ($\ref{7}$) into $\mathcal{I}(\alpha_k)$, we deduce that when $1\leq y\leq X$, one has
\begin{equation}\label{8}
\begin{aligned}
  \mathcal{I}(\alpha_k)=\dsum_{|g_{i_1}|\leq sX^{i_1}}\cdots\dsum_{|g_{i_l}|\leq sX^{i_l}}\doint I_{\mbd{g}}(\mbd{\gamma},y)\tilde{K}(\mbd{\gamma})d\mbd{\gamma},
\end{aligned}
\end{equation}
where 
\begin{equation}\label{9}
   I_{\mbd{g}}(\mbd{\gamma},y)=\doint\mathcal{F}_y(\alpha_k,\mbd{\alpha};\mbd{\gamma})e(-\mbd{\alpha}^{(l)}\cdot \mbd{g})d\mbd{\alpha}.
\end{equation}
By orthogonality, one finds that
\begin{equation}\label{10}
    \doint\mathcal{F}_y(\alpha_k,\mbd{\alpha};\mbd{\gamma})e(-\mbd{\alpha}^{(l)}\cdot \mbd{g})d\mbd{\alpha}=\dsum_{1\leq \mathbf{x}\leq 2X}\Delta(\alpha_k,\mbd{\gamma},\mbd{g},y),
\end{equation}
where $\Delta(\alpha_k,\mbd{\gamma},\mbd{g},y)$ is equal to 
$$e\biggl(\dsum_{i=1}^s(\alpha_k((x_{i}-y)^k-(x_{s+i}-y)^k)+\gamma_i(x_i-y)-\gamma_{s+i}(x_{s+i}-y))\biggr),$$
when 
\begin{equation}\label{2.102.10}
\dsum_{i=1}^s((x_i-y)^j-(x_{s+i}-y)^j)=h_j\ \textrm{with}\ 1\leq j\leq k-1,    
\end{equation}
in which $h_j=g_{j}$ when $j\in \{i_1,\ldots,i_l\}$, and $h_j=0$ when $j\notin \{i_1,\ldots,i_l\}$.
Otherwise, one finds that $\Delta(\alpha_k,\mbd{\gamma},\mbd{g},y)=0.$

By applying the Binomial Theorem within $(\ref{2.102.10})$, we have 
\begin{equation}\label{13}
    \dsum_{i=1}^s(x_i^j-x_{s+i}^j)=\dsum_{l=1}^{j}\binom{j}{l}h_ly^{j-l}\ \ \ (1\leq j\leq k-1),
\end{equation}
and
 \begin{equation}\label{15}
    \dsum_{i=1}^s(x_i^k-x_{s+i}^k)=\dsum_{l=1}^{k-1}\binom{k}{l}h_ly^{k-l}+\dsum_{i=1}^s((x_i-y)^k-(x_{s+i}-y)^k).
\end{equation} 

 By orthogonality, one infers from ($\ref{10}$),($\ref{13}$) and ($\ref{15}$) that by putting $h_k=0$
\begin{equation*}
\begin{aligned}
   \doint\mathcal{F}_y(\alpha_k,\mbd{\alpha};\mbd{\gamma})e(-\mbd{\alpha}^{(l)}\cdot \mbd{g})d\mbd{\alpha}= \omega_{y,\gamma}\doint \mathcal{F}_0(\alpha_k,\mbd{\alpha};\mbd{\gamma})e\biggl(-\dsum_{j=1}^k\alpha_{j}\biggl(\dsum_{ l=1}^j\binom{j}{l}h_ly^{j-l}\biggr)\biggr)d\mbd{\alpha},
\end{aligned}
\end{equation*}
where $\omega_{y,\mbd{\gamma}}=e(-\Gamma y)$ in which $\Gamma=\gamma_1+\cdots+\gamma_s-\gamma_{s+1}-\cdots-\gamma_{2s}.$ We now collect together terms corresponding to each power of $y$. On recalling $h_n=0$ when $n\notin \{i_1,\ldots,i_l\}$ and since by $j\leq k$,
    the highest degree of $y$ is $k-i_l$.
Furthermore, on recalling that $\alpha_j=0$ for $j\notin\{1,\ldots,k\}$ and the definition ($\ref{2.82.8}$) of $\delta_m,$ we find that
\begin{equation}\label{2.172.17}
\begin{aligned}
    \dsum_{ j=1}^k\alpha_{j}\biggl(\dsum_{ l=1}^j\binom{j}{l}y^{j-l}h_l\biggr)&=\dsum_{ m=0}^{k-i_l}\biggl(\dsum_{n=1}^l\alpha_{m+i_n}\binom{m+i_n}{i_n}h_{i_n}\biggr)y^{m}=\dsum_{ m=0}^{k-i_l}\delta_my^{m}.
\end{aligned}
\end{equation}
%where the inner sum on the right hand side is over $j$ and $l.$
%For concision, we write, for the coefficient of $y^{m}$, 
%\begin{equation}\label{2.14}
 %   \delta_{m}=\dsum_{\substack{1\leq l\leq j\leq k\\j-l=m}}\alpha_j\binom{j}{l}h_l.
%\end{equation}
Since $\alpha_{m+i_n}=0$ for $m+i_n>k,$ it is worth noting that no contribution arises from $n$ with $i_n> k-m$, in $\delta_m.$
%furthermore, on noting that the summation condition on the coefficient of $y^{k-m}$ ensures that no contribution arises when $l\geq m,$ we write
%$$=\dsum_{j=1}^k\biggl(\alpha_k\binom{k}{m}h_m+\dsum_{l=1}^{m-1}\dsum_{\substack{l\leq j\leq k-1\\ j-l=k-m}}\alpha_j\binom{j}{l}h_l\biggr)y^{k-m}=\dsum_{j=1}^k\delta_{k-m}y^{k-m},$$
%where $\delta_{k-m}=\alpha_k\displaystyle\binom{k}{m}h_m+\dsum_{\substack{1\leq l\leq j\leq k-1\\ j-l=k-m}}\alpha_j\binom{j}{l}h_l.$ 

From here, we are led from ($\ref{9}$) to the relation
\begin{equation*}
\begin{aligned}
&\dsum_{|g_{i_1}|\leq sX^{i_1}}\cdots\dsum_{|g_{i_l}|\leq sX^{i_l}} I_{\mbd{g}}(\mbd{\gamma},y)\\
&=\doint \mathcal{F}_0(\alpha_k,\mbd{\alpha};\mbd{\gamma})\dsum_{|h_{i_1}|\leq sX^{i_1}}\cdots\dsum_{|h_{i_l}|\leq sX^{i_l}}\omega_{y,\mbd{\gamma}}e\biggl(-\dsum_{ m=0}^{k-i_l}\delta_{m}y^{m}\biggr)d\mbd{\alpha}.
\end{aligned}
\end{equation*}
Since we took $y$ in $[1,X]$, we may conclude thus far
\begin{equation}\label{17}
\begin{aligned}
    X^{-1}\dsum_{1\leq y\leq X}\dsum_{|g_{i_1}|\leq sX^{i_1}}\cdots\dsum_{|g_{i_l}|\leq sX^{i_l}} I_{\mbd{g}}(\mbd{\gamma},y)=\doint \mathcal{F}_0(\alpha_k,\mbd{\alpha};\mbd{\gamma})\Xi(\alpha_k,\mbd{\alpha})d\mbd{\alpha}.
\end{aligned}
\end{equation}
%where 
%\begin{equation}{\label{18}}
%\begin{aligned}
 %  \Xi(\alpha_k,\mbd{\alpha}) =X^{-1}\dsum_{1\leq y\leq X}\dsum_{|h_{i_1}|\leq sX^{i_1}}\cdots\dsum_{|h_{i_l}|\leq sX^{i_l}}\omega_{y,\mbd{\gamma}}e\biggl(-\dsum_{ m=0}^{k-i_l}\delta_{m} y^{m}\biggr).
%\end{aligned}
%\end{equation}
%On recalling that $\omega_{y,\boldsymbol{\gamma}}=e(-\Gamma y),$ we may rewrite summands in ($\ref{18}$) as $e(-\sum_{m=0}^{k-i_l}\delta_{m}'y^{m}),$ where $\delta'_{n}=\delta_n\ (n\neq 1)$ and $\delta'_1=\delta_1+\Gamma.$

 Therefore, from ($\ref{8}$) and $(\ref{17})$, we conclude thus far that 
\begin{equation*}\label{2.17}
\begin{aligned}
   \mathcal{I}(\alpha_k)
   &\ll X^{-1}\dsum_{1\leq y\leq X}\mathcal{I}(\alpha_k)\\
   & =X^{-1}\dsum_{1\leq y\leq X}\dsum_{|g_{i_1}|\leq sX^{i_1}}\cdots\dsum_{|g_{i_l}|\leq sX^{i_l}}\doint I_{\mbd{g}}(\mbd{\gamma},y)\tilde{K}(\mbd{\gamma})d\mbd{\gamma}\\
   &=\doint \doint \mathcal{F}_0(\alpha_k,\mbd{\alpha};\mbd{\gamma})\Xi(\alpha_k,\mbd{\alpha})\tilde{K}(\mbd{\gamma})d\mbd{\alpha}d\mbd{\gamma}.
\end{aligned}
\end{equation*}
\end{proof}

%Here, notice that $\omega_{y,\mbd{\gamma}}$ is linear in $y$, and that $\mbd{\gamma}$ is fixed. Thus, by including this into the definition of $\delta_1,$ we may omit this $\omega_{y,\mbd{\gamma}}$ throughout the proof.
 
 \bigskip

We recall that
\begin{equation*}
     \Xi(\alpha_k,\mbd{\alpha}) =X^{-1}\dsum_{1\leq y\leq X}\dsum_{|h_{i_1}|\leq sX^{i_1}}\cdots\dsum_{|h_{i_l}|\leq sX^{i_l}}\omega_{y,\mbd{\gamma}}e\biggl(-\dsum_{ m=0}^{k-i_l}\delta_{m} y^{m}\biggr),
\end{equation*}
where 
\begin{equation*}
    \delta_{m}=\dsum_{n=1}^l\alpha_{m+i_n}\binom{m+i_n}{i_n}h_{i_n}.
\end{equation*}
%Furthermore, recall that 
%$|\alpha_k-a/q|\leq q^{-2}$ with $(q,a)=1$ in the hypothesis of Theorem 1.3.
We provide the upper bound for $\Xi(\alpha_k,\boldsymbol{\alpha})$ in terms of the denominator stemming from rational approximation to $\alpha_k$, by obtaining savings from all summations over $h_{i_1},\ldots, h_{i_l}.$
\begin{llll}
Suppose that $|\alpha_k-a/q|\leq q^{-2}$ with $(q,a)=1$. Then, for any $l$ with $1\leq l\leq k-t,$ we have
$$\Xi(\alpha_k,\boldsymbol{\alpha})\ll X^{i_1+\cdots+i_l+\epsilon}\left(\displaystyle\prod_{j=1}^l\left(q^{-1}+X^{-i_j}+{X^{-k+i_j}}+qX^{-k}\right)\right)^{1/((k-i_l)(k-i_l+1))}.$$
\end{llll}
In the proof of Lemma 2.3, we bound $\Xi(\alpha_k,\boldsymbol{\alpha})$ by mean value type estimates. Furthermore, we use Vinogradov's mean value theorem to deal with these mean value type estimates.
The argument described here is applicable to all possible arrangements of exponents $\mathbf{k}=(k_1,\ldots,k_t)$ with $t<k$. Especially, this argument is useful for the case $k_1-1=k_2$ and $t<k_1/2.$ Even for the case that $k_1-1>k_2$, experts will recognize that by taking $l=1$ the sum $\Xi(\alpha_k,\boldsymbol{\alpha})$ becomes the exponential sum with phase linear in $y$, and in this case a variant of our arguments coincides with the proof of [$\ref{ref12}$, Theorem 1.3] and [$\ref{ref15}$, Theorem 14.4].

%\begin{thmR}
%One has
%$$ \Xi(\alpha_k,\mbd{\alpha},\{c_m\}_{m},\{d_{jl}\}_{j,l}) \ll X^{i_1+i_2+\cdots+i_l+\epsilon}\left(X^{-2^{-K}}+\left[\dprod_{j=1}^{l}\left(\frac{1}{X^{i_j}}+\frac{1}{X^{k-i_j}}+\frac{1}{q}+\frac{q}{X^k}\right)\right]^{\frac{1}{(k-i_l)(k-i_l+1)}}\right),$$
%where $K=k-i_l-1.$
%\end{thmR}
%\begin{proof}
%We begin by using Weyl's differencing argument.

\begin{proof}[Proof of Lemma 2.3]
On recalling that $\omega_{y,\boldsymbol{\gamma}}=e(-\Gamma y),$ we may rewrite summands in $\Xi(\alpha_k,\boldsymbol{\alpha})$ as $e(-\sum_{m=0}^{k-i_l}\delta_{m}'y^{m}),$ where $\delta'_{n}=\delta_n\ (n\neq 1)$ and $\delta'_1=\delta_1+\Gamma.$

Define 
$$S^*(\mbd{\delta};X)=\sup_{I\subseteq [1,X]}\left|\dsum_{ y\in I}e\biggl(-\dsum_{1\leq m\leq k-i_l}\delta_{m}'y^{m}\biggr)\right|$$ where
$I$ runs over all intervals in $[1,X].$ In particular, we write $S(\boldsymbol{\delta};X)$ for the sum with $I=[1,X]$. Here and later, we put $2p=(k-i_l)(k-i_l+1).$ Define
\begin{equation}\label{19}
    \Upsilon_{p}(\mbd{\delta};X)=\dsum_{|h_{i_1}|\leq sX^{i_1}}\cdots\dsum_{|h_{i_l}|\leq sX^{i_l}}\left|S^*(\mbd{\delta};X)\right|^{2p}.
\end{equation}
Then, by applying H\"older's inequality to $\Xi(\alpha_k,\mbd{\alpha})$, we have
\begin{equation}\label{2.18}
\begin{aligned}
       \Xi(\alpha_k,\mbd{\alpha}) \leq X^{-1}(\Upsilon_{p}(\mbd{\delta};X))^{1/(2p)}X^{(i_1+\cdots+i_l)\left(1-1/(2p)\right)}.
\end{aligned}
\end{equation}

We first analyze $\Upsilon_{p}(\mbd{\delta};X)$. Define $\Omega(X)$ to be the box $A_{1}\times A_{2}\times\cdots\times A_{k-i_l},$ where
\begin{equation*}
    A_n:=A_n(\boldsymbol{\delta})=\{\theta_{n}\in [0,1): \|\delta_{n}'-\theta_{n}\|\leq 1/(4kX^{n})\}. 
\end{equation*}
Then, by [$\ref{ref5}$, Lemma 1], one infers that
\begin{equation}\label{21}
   S^*(\mbd{\delta};X)^{2p}
    \ll\displaystyle (\text{vol}(\Omega(X)))^{-1} \dint_{A_{1}}\dint_{A_{2}}\cdots\dint_{A_{k-i_l}} S^*(\mbd{\theta};X)^{2p}d\mbd{\theta}.
\end{equation}
 Recall the definition $\delta_n$ and the remark following ($\ref{2.172.17}$). Then, we see that $\delta_{k-i_j}$ is a linear combination of $h_{i_j},\ldots,h_{i_l}$. We define the quantity $H_l(\boldsymbol{\theta})$ to be the number of solutions $(h_{i_1},h_{i_2},\ldots,h_{i_l})$ with $|h_{i_j}|\leq sX^{i_j}$ of the system
\begin{equation*}
     \|\delta_{n}'-\theta_{n}\|\leq 1/(4kX^{n})\ \ \ \ (n=k-i_1,k-i_2,\ldots,k-i_l),
\end{equation*}
and put
\begin{equation*}
H_l=\sup_{\boldsymbol{\theta}\in [0,1)^l}H_l(\boldsymbol{\theta}).    
\end{equation*}
%Therefore, the equation (0.19) can be bounded above by
%%\begin{equation}
    %\dint_0^1\cdots\dint_0^1\displaystyle\left(\prod_{j=1}^{l}\dsum_{|h_{i_j}|\leq sX^{i_j}}\chi_{A_{i_j}}(\theta_{i_j})\right) \sup_{I\subseteq [1,X]}\left|\dsum_{y\in I}e\left(\dsum_{i_l\leq m\leq k-1}\theta_m y^{k-m}\right)\right|^{(k-i_l)(k-i_l+1)}d\theta_{k-1}\cdots d\theta_{i_l}
%\end{equation}
Therefore, on substituting ($\ref{21}$) into ($\ref{19}$), and expanding $A_j$ to $[0,1)$ for 
$$j\notin \{k-i_1,k-i_2\ldots,k-i_l\},$$ we obtain the bound
\begin{equation*}
\begin{aligned}
     &\Upsilon_{p}(\mbd{\delta};X)\\
    &\ll\displaystyle(\text{vol}(\Omega(X)))^{-1}\dint_0^1\cdots\dint_0^1\dsum_{|h_{i_1}|\leq sX^{i_1}}\cdots\dsum_{|h_{i_l}|\leq sX^{i_l}} \dint_{A_{k-i_1}}\dint_{A_{k-i_2}}\cdots\dint_{A_{k-i_l}} S^*(\mbd{\theta};X)^{2p}d\mbd{\theta}.\\
    \end{aligned}
\end{equation*}
Since $(\text{vol}(\Omega(X)))^{-1}=X^{1+\cdots+(k-i_l)}$ and by the definition of $H_l$, we infer that
\begin{equation}\label{22}
\begin{aligned}
    \Upsilon_{p}(\mbd{\delta};X)\ll& X^{1+\cdots+(k-i_l)}H_l\dint_0^1\cdots\dint_0^1S^*(\mbd{\theta};X)^{2p}d\mbd{\theta},
\end{aligned}
\end{equation}

%To bound $H_l$, we utilize a classical lemma [$\ref{ref8}$, Lemma 6]. We record this lemma without the proof.
%\begin{llll}
%Let $\alpha$ have an approximation $|\alpha-p/q|\leq 1/q^2,(p,q)=1$. Then the number of solutions of the diophantine inequality
%$$\|\alpha x+\beta\|\leq 1/Y,\ \ |x|\leq X,\ \ x\in\Z$$
%does not exceed $4(1+q/Y)(1+X/q).$
%\end{llll}

To bound $H_l$, we first analyse $H_l(\boldsymbol{\theta})$. Recall again the definition $\delta_m$ and the remark following ($\ref{2.172.17}$). Then, we have
$$ \delta_{k-i_j}=\alpha_k\binom{k}{i_j}h_{i_j}+\dsum_{n=j+1}^{l}\alpha_{k-i_j+i_n}\binom{k-i_j+i_n}{i_n}h_{i_n},$$
for all $j=1,\ldots, l.$ %starting with $m=i_1$ and proceeding in decreasing order. If $m=i_1$ and fix $h_{i_2}, ...,h_{i_{l}}$, then $\delta_{k-i_1}=\alpha_k\binom{k}{i_1}h_{i_1}+\cdots$ is linear in $h_{i_1}$. Thus, by Lemma 2.2, and the approximation hypothesis for $\alpha_k$ 
Recall that $\delta'_{k-i_j}=\delta_{k-i_j}+\Gamma $ for $k-i_j= 1$, and $\delta'_{k-i_j}=\delta_{k-i_j},$ otherwise. Meanwhile, by [$\ref{ref5}$, Lemma 3],  when $m\in \N$, $\alpha,\beta\in \R$ and $|\alpha-a/q|\leq q^{-2}$, the number of solutions of
$$\|m\alpha x+\beta\|\leq 1/Y,$$
with $|x|\leq X$, is at most $(1+4q/Y)(1+4mX/q).$ Put $\alpha=\alpha_k$ with $|\alpha_k-a/q|\leq q^{-2}$, $m=\binom{k}{i_j},$ $X=sX^{i_j}$, $Y=4kX^{k-i_j}.$  Then, for fixed $h_{i_{j+1}},\ldots, h_{i_l}$, the number of $h_{i_j}$ of 
$$\|\delta_{k-i_j}'-\theta_{k-i_j}\|\leq 1/(4kX^{k-i_j}),$$
with $|h_{i_j}|\leq sX^{i_j},$ is at most $\ll X^{i_j}(q^{-1}+X^{-i_j}+X^{-(k-i_j)}+qX^{-k}).$
If we proceed this in descending order $j=l,l-1,\ldots,1$, we infer that 
%Now we fix $h_{i_1},h_{i_{3}},...,h_{i_l}$, and look at the inequality with $m=i_{2}$. Then, again $\delta_{k-i_2}=\alpha_k\binom{k}{i_2}h_{i_{2}}+\cdots$ is linear in $h_{i_{2}}$. Hence, Lemma 2.3 can be used again, obtaining at most $\ll X^{i_{2}}(q^{-1}+X^{-i_{2}}+X^{-(k-i_{2})}+qX^{-k})$. Proceeding in this way, one has
\begin{equation}\label{27}
    H_l(\boldsymbol{\theta})\ll X^{i_1+i_2+\cdots+i_l}\displaystyle\prod_{j=1}^l\left(q^{-1}+X^{-i_j}+{X^{-k+i_j}}+qX^{-k}\right).
\end{equation}
By taking supremum over $\boldsymbol{\theta},$ we may replace $H_l(\boldsymbol{\theta})$ with $H_l$ in $(\ref{27}).$
For concision, we write 
\begin{equation}\label{2.242.24}
    R_l=\displaystyle\prod_{j=1}^l\left(q^{-1}+X^{-i_j}+{X^{-k+i_j}}+qX^{-k}\right)
\end{equation}
Therefore, from ($\ref{22}$) and ($\ref{27}$), one has by applying the Carleson-Hunt theorem [$\ref{ref24}$]
\begin{align*} 
\Upsilon_{p}(\mbd{\delta};X)&\ll X^{1+\cdots+(k-i_l)}H_l\doint S^*(\mbd{\theta};X)^{2p}d\mbd{\theta}\\
&\ll X^{i_1+i_2+\cdots+i_l}R_lX^{1+\cdots+(k-i_l)}\doint S(\mbd{\theta};X)^{2p}d\mbd{\theta}.
\end{align*}
 Hence, by Vinogradov's mean value theorem, the last expression is 
$O(X^{(2p+\epsilon)}X^{i_1+i_2+\cdots+i_l}R_l).$
Consequently, by ($\ref{2.18}$), we see that
 \begin{equation}\label{28}
     \Xi(\alpha_k,\mbd{\alpha})\ll X^{i_1+i_2+\cdots+i_l+\epsilon}R_l^{1/(2p)}.
 \end{equation}
 On recalling the definition $R_l$, we complete the proof of Lemma 2.3.
%Thus, we complete the proof of Proposition 1.
%\end{proof}

\end{proof}

\subsection{Proof of Theorem 1.3}

\begin{proof}
We combine all lemmas in section 2.1 to prove Theorem 1.3. On recalling ($\ref{2.242.24}$) and $2p=(k-i_l)(k-i_l+1)$, by Lemma 2.2 and Lemma 2.3, we have 
\begin{equation}\label{2.272.27}
    \mathcal{I}(\alpha_k)\ll X^{i_1+\cdots+i_l+\epsilon}R_l^{1/(2p)}\doint \doint \mathcal{F}_0(\alpha_k,\mbd{\alpha};\mbd{\gamma})\tilde{K}(\mbd{\gamma})d\mbd{\alpha}d\mbd{\gamma}.
\end{equation}
%Note that by using the bound in Lemma 2.1, we have
%$$\doint \mathcal{F}_0(\alpha_k,\mbd{\alpha};\mbd{\gamma})\Xi(\alpha_k,\mbd{\alpha})d\boldsymbol{\alpha}\ll X^{i_1+i_2+\cdots+i_l+\epsilon}R_l^{1/(2p)}\doint \mathcal{F}_0(\alpha_k,\mbd{\alpha};\mbd{\gamma})d\boldsymbol{\alpha}.$$

Meanwhile, by applying the H\"older's inequality and a change of variable, one sees that 
\begin{equation}\label{2.282.28}
\doint \mathcal{F}_0(\alpha_k,\mbd{\alpha};\mbd{\gamma})d\boldsymbol{\alpha}\leq \sup_{\gamma\in [0,1)}\doint |f_0(\alpha_k,\mbd{\alpha};\gamma)|^{2s}d\boldsymbol{\alpha}=\doint |f(\alpha_k,\mbd{\alpha})|^{2s}d\boldsymbol{\alpha}.    
\end{equation}
Furthermore, on recalling $(\ref{2.62.6}),$ we find that 
$$
\dint_0^1|K(\gamma)|d\gamma\leq \dint_0^1 \textrm{min}\{X,\|\gamma\|^{-1}\}d\gamma\ll \log X,$$
and hence
\begin{equation}\label{2.292.29}
\doint|\tilde{K}(\mbd{\gamma})|d\mbd{\gamma}\ll (\log X)^{2s}.
\end{equation}
%Therefore, we have
%$$\doint \mathcal{F}_0(\alpha_k,\mbd{\alpha};\mbd{\gamma})\Xi(\alpha_k,\mbd{\alpha})d\boldsymbol{\alpha}\ll X^{i_1+i_2+\cdots+i_l+\epsilon}R_l^{1/(2p)}\doint |f(\alpha_k,\mbd{\alpha})|^{2s}d\boldsymbol{\alpha}.$$

On substituting ($\ref{2.282.28}$) and ($\ref{2.292.29}$) into the right hand side in ($\ref{2.272.27}$), we find that
%\begin{equation}\label{2.25}
 %   \begin{aligned}
  %  &\doint|F(\alpha_{k_1},\mbd{\alpha}^{t-1};\mbd{\beta}^{k-t-l})|^{2s} d\mbd{\beta}^{k-t-l}d\mbd{\alpha}^{t-1}\\
  %  &\ll X^{i_1+i_2+\cdots+i_l+\epsilon}R_l^{1/(2p)}\doint \left|f(\alpha_k,\mbd{\alpha})\right|^{2s}d\mbd{\alpha}\doint|\tilde{K}(\mbd{\gamma})|d\mbd{\gamma}.
   % \end{aligned}
%\end{equation}
$$ \mathcal{I}(\alpha_k)\ll X^{i_1+\cdots+i_l+\epsilon}R_l^{1/(2p)}\doint |f(\alpha_k,\mbd{\alpha})|^{2s}d\boldsymbol{\alpha}.$$
Therefore, we conclude from Lemma 2.1 that 
$$\doint|F(\alpha_{k_1},\mbd{\alpha}^{t-1})|^{2s}d\mbd{\alpha}^{t-1}\ll R_l^{1/(2p)} X^{i_1+i_2+\cdots+i_{k-t}+\epsilon}\doint \left|f(\alpha_k,\mbd{\alpha})\right|^{2s}d\mbd{\alpha}.$$
\end{proof}

\section{Proof of Theorem 1.4}

In this section, we provide the proof of Theorem 1.4. In the previous section, we obtained the mean value over all coefficients but the leading coefficient. Thus, Theorem 1.4 follows by integrating over $\alpha_k$ lying on each of major arcs and minor arcs. To be specific, minor arcs estimates in Theorem 1.4 $(\romannumeral2)$ follow immediately from Theorem 1.3 and Diophantine approximation of the leading coefficient. For major arc estimates in Theorem 1.4 $(\romannumeral1)$, we use a consequence of [$\ref{ref14}$, Theorem 14.4] with applications of H\"older's inequality. 

\bigskip

\begin{proof}[Proof of Theorem 1.4]

It follows from ($\ref{28}$) with $2p=(k-i_l)(k-i_l+1)$ that whenever $|\alpha_k-a/q|\leq q^{-2},$ one has
 \begin{equation}\label{30}
 \begin{aligned}
     \Xi(\alpha_k,\mbd{\alpha})&\ll X^{i_1+i_2+\cdots+i_l+\epsilon}\left(\displaystyle\prod_{j=1}^l\left(q^{-1}+X^{-i_j}+{X^{-k+i_j}}+qX^{-k}\right)\right)^{1/(2p)}\\
     &\ll X^{i_1+i_2+\cdots+i_l+\epsilon}\left(q^{-1}+X^{-1}+q{X^{-k}}\right)^{\sigma}, 
 \end{aligned}
 \end{equation}
 where 
 \begin{align*}
\sigma= \frac{l}{(k-i_l)(k-i_l+1)}.
\end{align*}

We first provide estimates for the major arcs. Assume that $\alpha_k\in \mathfrak{M}$. Note that transference principle [$\ref{ref28}$, Theorem 14.1] tells that whenever we have a function $\Psi:\R\rightarrow \C$ with the upper bound
$$\Psi(\alpha)\ll X(q^{-1}+Y^{-1}+qZ^{-1})^{\theta},$$
where $\theta,X,Y,Z$ are positive real numbers, and $a\in \Z$, $q\in \N$ satisfying $(a,q)=1$ and $|\alpha-a/q|\leq q^{-2}$, then we deduce that $$\Psi(\alpha)\ll X(\lambda^{-1}+Y^{-1}+\lambda Z^{-1})^{\theta},$$
with $\lambda=r+Z|r\alpha-b|$, and $b\in \Z$, $r\in \N$ satisfying $(b,r)=1.$
Therefore, %by recalling the definition of $\mathfrak{M}$% and by applying transference principle [$\ref{ref28}$, Theorem 14.1]
 one infers from $(\ref{30})$ that whenever $b\in \Z$ and $r\in\N$ satisfy $(b,r)=1$ and $|\alpha_k-b/r|\leq r^{-2}$, then it follows that
 \begin{equation*}
     \Xi(\alpha_k,\mbd{\alpha})\ll X^{i_1+i_2+\cdots+i_l+\epsilon}(\lambda^{-1}+X^{-1}+\lambda X^{-k})^{\sigma},
 \end{equation*}
 where $\lambda=r+X^k|r\alpha_k-b|.$ Moreover, when $\alpha_k\in \mathfrak{M}(r,b)\subseteq\mathfrak{M},$ one has $r\leq X$ and $X^k|r\alpha_k-b|\leq X$, so that $\lambda\leq 2X.$ Therefore, we see from it that 
 one has 
  \begin{equation*}
     \Xi(\alpha_k,\mbd{\alpha})\ll X^{i_1+i_2+\cdots+i_l+\epsilon}\Psi(\alpha_k),
 \end{equation*}
 where $\Psi(\alpha_k)$ is the function taking the value $(q+X^k|q\alpha_k-a|)^{-\sigma},$ when one has $\alpha_k\in \mathfrak{M}(q,a)\subseteq \mathfrak{M},$ otherwise $\Psi(\alpha_k)=0.$
Hence, one has
\begin{equation}\label{eq373737}
\begin{aligned}
    \dint_{\mathfrak{M}}\doint |f(\alpha_k,\mbd{\alpha})|^{2s} \Xi(\alpha_k,\mbd{\alpha})d\mbd{\alpha}d\alpha_k\ll X^{i_1+\cdots+i_l+\epsilon}\dint_{\mathfrak{M}}\doint |f(\alpha_k,\mbd{\alpha})|^{2s} \Psi(\alpha_k)d\mbd{\alpha}d\alpha_k.
\end{aligned}    
\end{equation}

Let us first assume that $2s\geq k(k+1).$ Then, since $\Psi(\alpha_k)\leq 1$, one finds that by Vinogradov's mean value theorem
\begin{equation}\label{32'}
    \dint_{\mathfrak{M}}\doint |f(\alpha_k,\mbd{\alpha})|^{2s} \Psi(\alpha_k)d\mbd{\alpha}d\alpha_k\ll \dint_{0}^1\doint |f(\alpha_k,\mbd{\alpha})|^{2s} d\mbd{\alpha}d\alpha_k\ll X^{2s-k(k+1)/2+\epsilon}.
\end{equation}

Next, let us assume that $k^2+(1-2\sigma)k+2\sigma\leq 2s< k(k+1)$.
By applying H\"older's inequality, one obtains that
\begin{equation}\label{32}
\begin{aligned}
   & \dint_{\mathfrak{M}}\doint |f(\alpha_k,\mbd{\alpha})|^{2s} \Psi(\alpha_k)d\mbd{\alpha}d\alpha_k\\
  &  \ll \left(\dint_{\mathfrak{M}}\doint |f(\alpha_k,\mbd{\alpha})|^{2s_0} \Psi(\alpha_k)^{\frac{1}{\sigma}}d\mbd{\alpha}d\alpha_k\right)^{\sigma}\left(\dint_{\mathfrak{M}}\doint |f(\alpha_k,\mbd{\alpha})|^{k(k+1)} d\mbd{\alpha}d\alpha_k\right)^{1-\sigma},
\end{aligned}
\end{equation}
with $s_0=(2s-k(k+1)(1-\sigma))/(2\sigma)$. Notice from the range of $2s$ that $k(k-1)\leq 2s_0\leq k(k+1)$.%Notice that $2s_0\sigma+k(k+1)(1-\sigma)=2s,$ and so $2s\leq 2s_0\leq k(k+1).$ Hence, on observing that $\sigma\leq 1$, we may assume that $k(k-1)\leq 2s_0\leq k(k+1)$.

As a consequence of [$\ref{ref6}$, Lemma 2], one finds that when $2s_0$ is an even number
\begin{equation}\label{3.63.6}
    \dint_{\mathfrak{M}}\doint |f(\alpha_k,\mbd{\alpha})|^{2s_0} \Psi(\alpha_k)^{\frac{1}{\sigma}}d\mbd{\alpha}d\alpha_k\ll X^{\epsilon-k}(XI_1+I_2),
\end{equation}
where 
\begin{equation*}
    I_1=\dint_0^1\doint|f(\alpha_k,\mbd{\alpha})|^{2s_0}d\mbd{\alpha}d\alpha_k,\ \textrm{and}\  I_2=\doint |f(0,\mbd{\alpha})|^{2s_0}d\mbd{\alpha}.
\end{equation*}
By Vinogradov's mean value theorem, whenever $k(k-1)\leq 2s_0\leq k(k+1),$ we have $I_1\ll X^{s_0+\epsilon}.$
On the other hands, when $2s_0\geq k(k-1),$ we have $  I_2\ll X^{2s_0-k(k-1)/2+\epsilon}.$
Thus, for all even numbers $2s_0$ with $k(k-1)\leq 2s_0\leq k(k+1)$, we find from ($\ref{3.63.6}$) that
\begin{equation}\label{33}
\begin{aligned}
    \dint_{\mathfrak{M}}\doint |f(\alpha_k,\mbd{\alpha})|^{2s_0} \Psi(\alpha_k)^{\frac{1}{\sigma}}d\mbd{\alpha}d\alpha_k\ll X^{s_0-k+1+\epsilon}+X^{2s_0-k(k+1)/2+\epsilon}.
\end{aligned}
\end{equation}
Notice here that the situation that two terms of the bound in ($\ref{33}$) are same occurs when $2s_0=k^2-k+2$, which is an even number. Thus, by interpolation between even numbers $2s_0$, one finds that ($\ref{33}$) also holds for any real numbers $2s_0$ between $k(k-1)$ and $k(k+1).$ On substituting ($\ref{33}$) into ($\ref{32}$) and applying Vinogradov's mean value theorem, one has
\begin{equation*}
\begin{aligned}
     \dint_{\mathfrak{M}}\doint |f(\alpha_k,\mbd{\alpha})|^{2s} \Psi(\alpha_k)d\mbd{\alpha}d\alpha_k&\ll \left(X^{s_0-k+1+\epsilon}+X^{2s_0-k(k+1)/2+\epsilon}\right)^{\sigma}(X^{k(k+1)/2})^{1-\sigma}.
\end{aligned}
\end{equation*}
Since we have $2s_0\sigma+k(k+1)(1-\sigma)=2s,$ this bound is seen to be 
$$ X^{s-\sigma(k-1)+\epsilon}+X^{2s-k(k+1)/2+\epsilon}.$$
Furthermore, since $2s\geq k^2+(1-2\sigma)k+2\sigma,$ this bound can be replaced by $O(X^{2s-k(k+1)/2+\epsilon}).$
Thus, one concludes that
whenever $ k^2+(1-2\sigma)k+2\sigma\leq 2s<k(k+1)$
\begin{equation}\label{36'}
    \dint_{\mathfrak{M}}\doint |f(\alpha_k,\mbd{\alpha})|^{2s} \Psi(\alpha_k)d\mbd{\alpha}d\alpha_k\ll X^{2s-k(k+1)/2+\epsilon}.
\end{equation}

Thus, by ($\ref{eq373737}$), ($\ref{32'}$) and ($\ref{36'}$), whenever $2s\geq k^2+(1-2\sigma)k+2\sigma$ we find that
\begin{equation*}
    \dint_{\mathfrak{M}}\doint |f(\alpha_k,\mbd{\alpha})|^{2s} \Xi(\alpha_k,\mbd{\alpha})d\mbd{\alpha}d\alpha_k\ll X^{i_1+\cdots+i_l+\epsilon}X^{2s-k(k+1)/2+\epsilon}.
\end{equation*}
Then, on recalling the definition of $\mathcal{F}_0(\alpha_k,\mbd{\alpha};\mbd{\gamma})$, it follows from H\"older's inequality and a change of variable that 
\begin{equation*}
\dint_{\mathfrak{M}} \doint \mathcal{F}_0(\alpha_k,\mbd{\alpha};\mbd{\gamma})\Xi(\alpha_k,\mbd{\alpha})d\mbd{\alpha}d\alpha_k\ll X^{i_1+\cdots+i_l+\epsilon}X^{2s-k(k+1)/2+\epsilon}.
\end{equation*}
Consequently, combining this with Lemma 2.1 and Lemma 2.2, we deduce that
\begin{align*}
   &\dint_{\mathfrak{M}}\doint\left|F(\alpha_{k_1},\mbd{\alpha}^{t-1})\right|^{2s}d\mbd{\alpha}^{t-1}d\alpha_{k_1}\\
   &\ll X^{i_{l+1}+\cdots+i_{k-t}}\dint_{\mathfrak{M}}\mathcal{I}(\alpha_k)d\alpha_k\\
    &\ll X^{i_{l+1}+\cdots+i_{k-t}}\dint_{\mathfrak{M}}\doint\doint\mathcal{F}_0(\alpha_k,\mbd{\alpha};\mbd{\gamma})\Xi(\alpha_k,\mbd{\alpha})\tilde{K}(\boldsymbol{\gamma})d\mbd{\alpha}d\boldsymbol{\gamma}d\alpha_k\\
    &\ll X^{2s-D+\epsilon},
\end{align*}
where we have used $(\ref{2.292.29})$.

 Next, we provide estimates for the minor arcs. When $\alpha_k\in \mathfrak{m}$, there exists $q$ and $a$ with $(q,a)=1$ such that $|\alpha_k-a/q|\leq (2k)^{-1}q^{-1}X^{-k+1}$ with $X<q<X^{k-1}.$ Thus, on recalling ($\ref{30}$), when $\alpha_k\in \mathfrak{m},$ we deduce that $\Xi(\alpha_k,\mbd{\alpha})\ll X^{i_1+\cdots+i_{k-t}-\sigma+\epsilon}.$
 Therefore, by applying Theorem 1.3 together with Vinogradov's mean value theorem, whenever $2s\geq k_1(k_1+1)$ one has 
\begin{align*}
    \dint_{\mathfrak{m}}\doint\left|F(\alpha_{k_1},\mbd{\alpha}^{t-1})\right|^{2s}d\mbd{\alpha}d\alpha_{k_1}&\ll X^{i_1+\cdots+i_{k-t}-\sigma+\epsilon}\dint_{0}^1\doint\left|f(\alpha_k,\mbd{\alpha})\right|^{2s}d\mbd{\alpha}d\alpha_{k_1}\\
     &\ll X^{2s-D-\sigma+\epsilon}.
\end{align*}
Therefore, by taking $l$ that maximizes the exponent $\sigma,$ the conclusion of Theorem 1.4 follows. 
\end{proof}

\bigskip

\section{Proof of Theorem 1.1}

\bigskip
In this section, we provide Theorem 4.1, which is more quantitative than Theorem 1.1. It is worth noting that Theorem 1.1 immediately follows from Theorem 4.1.

The main ingredients of the proof in this section are the arguments in [$\ref{ref13}$, Theorem 1.3]. Wooley [$\ref{ref13}$, Theorem 1.3] provided upper bounds for exponential sums by bounding the pointwise estimates by mean value estimates over major and minor arcs. Meanwhile, a classical way widely used in studying fractional parts of polynomial is closely related to the upper bounds of associated exponential sum. Thus, we exploit the argument in [$\ref{ref13}$] to obtain upper bounds of associated exponential sums in terms of mean values of exponential sums. Thus, upper bounds for these mean values of exponential sums deliver the conclusion of Theorem 4.1.  

\bigskip

\begin{te}
Let $\epsilon>0$ and $s,k$ be natural numbers with $k\geq 6.$ Suppose that $X$ is sufficiently large in terms of $s,k$ and $\epsilon$. Consider $\alpha_i\in\R$ with $1\leq i\leq s.$ Then, for $s\geq k+2$ one has
 $$\min_{\substack{0\leq \boldsymbol{x}\leq X\\ \boldsymbol{x}\neq \boldsymbol{0}}}\|\alpha_1x_1^k+\alpha_2x_2^k+\cdots+\alpha_sx_s^k\|\leq X^{-\sigma(s,k)+\epsilon},$$
 where 
 \begin{equation*}
 \sigma(s,k)=\min\biggl\{\frac{s}{k(k+1)-s}, 1\biggr\}.
 \end{equation*}
 \end{te}

\begin{proof}[Proof of Theorem 1.1]
Note that whenever $s\geq k(k+1)/2$ the exponent $\sigma(s,k)$ in Theorem 4.1 becomes $1$. Therefore, Theorem 1.1 immediately follows from Theorem 4.1.
\end{proof}

\subsection{Outline of the proof of Theorem 4.1}

 We provide outline of the proof of Theorem 4.1. We begin with stating a classical lemma from the theory of fractional parts of polynomials [$\ref{ref2}$, Theorem 2.2], which relates fractional parts of a sequence of real numbers to the associated exponential sum.
 
\begin{llll}
Let $x_1,\ldots,x_N$ be real numbers. Suppose that $\|x_n\|\geq H^{-1}$ for every $n$ with $1\leq n\leq N$. Then, 
\begin{equation*}
    \dsum_{1\leq h\leq H}\bigl|\dsum_{n=1}^Ne(hx_n)\bigr|\gg N.
\end{equation*}
\end{llll}

 Let $H$ be a positive number with $H\leq X^{1-\nu}$ for sufficiently small $\nu>0$. Suppose that 
\begin{equation}\label{4.14.1}
    \min_{\substack{0\leq \boldsymbol{x}\leq X\\ \boldsymbol{x}\neq \boldsymbol{0}}}\|\alpha_1x_1^k+\alpha_2x_2^k+\cdots+\alpha_sx_s^k\| > H^{-1}.
\end{equation}
Then, by Lemma 4.2, we have
\begin{equation}\label{34}
    \dsum_{1\leq h\leq H}\bigl|\dsum_{1\leq \boldsymbol{x}\leq X}e(h(\alpha_1x_1^k+\alpha_2x_2^k+\cdots+\alpha_sx_s^k))\bigr|\gg X^{s}.
\end{equation}

For concision, here and throughout, we write $[1,H]=[1,H]\cap \Z.$   %is defined by
%\begin{equation*}
 %   \mathfrak{M}=\bigcup_{\substack{1\leq a\leq q \leq X\\(q,a)=1}}\mathfrak{M}(q,a),
%\end{equation*}
%where $\mathfrak{M}(q,a)=\{\alpha\in [0,1)|\ |q\alpha-a|\leq (2k)^{-1}X^{1-k}\}$, and the minor arc is defined by $\mathfrak{m}=[0,1)\setminus \mathfrak{M}.$ 
Recall the definition ($\ref{eq1.6}$) of $\mathfrak{M}$ and $\mathfrak{m}$. On observing that each real number $h\alpha_j$ lies either on $\mathfrak{M}$ or $\mathfrak{m},$ one can decompose the set $ [1,H]$ into $2^{s}$ sets, $H_1,\ldots,H_{2^s}$, such that the set $\{h\alpha_j|\ h\in H_i\}\subseteq\mathfrak{M}$ or $\{h\alpha_j|\ h\in H_i\}\subseteq\mathfrak{m}$, for all $1\leq j\leq s$ and $1\leq i\leq 2^s.$ Our goal is to show that for every $H_i\ (i=1,\ldots,2^s)$, we have
\begin{equation}\label{35}\dsum_{h\in H_i}\bigl|\dsum_{1\leq \boldsymbol{x}\leq X}e(h(\alpha_1x_1^k+\alpha_2x_2^k+\cdots+\alpha_sx_s^k))\bigr|\ll X^{s-\eta}\ \textrm{for some}\ \eta=\eta(k,\nu)>0,
\end{equation}
which contradicts ($\ref{34}$) for sufficiently large $X$ in terms of $\eta$ and $s$. Thus, this forces us to conclude that for sufficiently large $X$, we have
\begin{equation*}
    \min_{\substack{0\leq \boldsymbol{x}\leq X\\ \boldsymbol{x}\neq \boldsymbol{0}}}\|\alpha_1x_1^k+\alpha_2x_2^k+\cdots+\alpha_sx_s^k\| \leq H^{-1}.
\end{equation*}
Therefore, by letting $\nu\rightarrow 0$, we are done to prove Theorem 4.1.

% Let us temporarily assume that for every $h\in H_1$, the coefficient $h\alpha_1,..,h\alpha_{m}$ are in $\mathfrak{M}$, and $h\alpha_{m+1},...,h\alpha_{s}$ are in $\mathfrak{m}$. Observe that by Holder's inequality, if possible,
%\begin{equation}\label{36}
%\begin{aligned}
 %   \dsum_{h\in H_1}|\dsum_{1\leq x_1,x_2,...,x_s\leq X}e(h(\alpha_1x_1^k+\alpha_2x_2^k+\cdots+\alpha_sx_s^k))|\\
 %  \ll (\dsum_{h\in H_1}1)^{1-\frac{km+s-m}{k(k+1)}}A_1^{\frac{1}{k+1}}\cdots A_m^{\frac{1}{k+1}}B_{m+1}^{\frac{1}{k(k+1)}}\cdots B_{s}^{\frac{1}{k(k+1)}},
 %   \end{aligned}
%\end{equation}
%where $A_i=\dsum_{h\in H_1}|\dsum_{1\leq x_i\leq X}e(h\alpha_ix_i^k)|^{k+1}$ and $B_i=\dsum_{h\in H_1}|\dsum_{1\leq x_i\leq X}e(h\alpha_ix_i^k)|^{k(k+1)}$. Thus, by obtaining the upper bounds of $A_i$ and $B_i$, we shall show ($\ref{35}$). %By relating $A_i$ and $B_i$ to the mean value over major and minor arcs, respectively, one will get (57). 

\bigskip

\subsection{Preliminary manoeuvre}
Under the assumption ($\ref{4.14.1}$), we can obtain extra information about $\alpha_1,\ldots,\alpha_s$. In order to describe this information, we must define $\mathfrak{M}^{H}$ by
$$\mathfrak{M}^{H}=\bigcup_{\substack{0\leq a\leq q \leq X\\(q,a)=1}}\mathfrak{M}^{H}(q,a),$$
where $\mathfrak{M}^H(q,a)=\left\{\alpha\in[0,1):\ \left|q\alpha-a\right|<X^{1-k}H^{-1}\right\}.$ Define $\mathfrak{m}^H$ by $[0,1)\setminus \mathfrak{M}^H$. Note that if there exists $\alpha_j$ contained in $\mathfrak{M}^H$, it follows by putting $x_j=q$ and $x_i=0\ (i\neq j)$ that 
$$\min_{\substack{0\leq \boldsymbol{x}\leq X\\ \boldsymbol{x}\neq \boldsymbol{0}}}\|\alpha_1x_1^k+\alpha_2x_2^k+\cdots+\alpha_sx_s^k\|\leq\min_{1\leq x_j\leq X}\|\alpha_jx_j^k\|\leq \|\alpha_j q^k\|\leq q^{k-1}\|\alpha_jq\|\leq H^{-1},$$
which contradicts ($\ref{4.14.1}$). Hence, from the assumption ($\ref{4.14.1}$), we may assume that all $\alpha_j\ (j=1,\ldots,s)$ are in $\mathfrak{m}^H.$

Furthermore, whenever $\alpha_j\in \mathfrak{m}^H$ with $H\leq X^{1-\nu}$ for sufficiently small $\nu>0,$ one has for all $h\in [1,H]\cap \Z$
\begin{equation}\label{46}
    \dsum_{1\leq x\leq X}e(h\alpha_j x^k)\ll X^{1-\delta_1}
\end{equation}
for some positive number $\delta_1=\delta_1(k,\nu)$. Indeed, suppose that there exists $h\in H$ such that
\begin{equation*}
    \dsum_{1\leq x\leq X}e(h\alpha_j x^k)\geq X^{1-\delta_1}.
\end{equation*}
Then, the Weyl's inequality [$\ref{ref10}$, Lemma 2.4] readily confirms that 
there exist $q\in \N$ and $a\in \Z$ such that $q<X^{\eta}$ and 
$$|h\alpha_j-a/q|\leq q^{-1}X^{\eta-k},$$
where $\eta=\eta(\delta_1).$ This gives
$$|\alpha_j-a/(qh)|\leq (qh)^{-1}X^{\eta-k}.$$
For sufficiently small $\delta_1>0$ so that $\eta=\eta(\delta_1)$ is smaller than $\nu,$ one has $qh<X^{\eta}X^{1-\nu}<X$ and
$$|\alpha_j-a/(qh)|\leq (qh)^{-1}X^{1-k}H^{-1}.$$
This yields that $\alpha_j\in \mathfrak{M}^H$, which contradicts $\alpha_j\in \mathfrak{m}^H.$
%Note also that if $\alpha_j\in\mathfrak{m}_H$ with $H\leq X^{1-\epsilon}$, there exist $q$ and $a$ with $(q,a)=1$ such that $X\leq q \leq X^{k-1}H$ and
%\begin{equation*}
%  |\alpha_j-a/q|<q^{-1}X^{-k+1}H^{-1}.
%\end{equation*}
%Thus, there exists $q_1$ and $a_1$ with $(q_1,a_1)=1$ such that  $q_1> X^{\epsilon}$ and $|h\alpha_j-a_1/q_1|<q_1^{-1}X^{-k+1}$, since $h\leq X^{1-\epsilon}$.
 %Thus, by Weyl's inequality [$\ref{ref10}$, Lemma 2.4], for every $h\in [1,H]\cap \Z$, we have

%or a positive number $\delta_1=\delta_1(\epsilon,k)$ less than $2^{(1-k)\epsilon}.$ 
\bigskip

\subsection{Lemma and proposition}

To prove ($\ref{35}$), we require arguments used in [$\ref{ref13}$, Theorem 1.3], which relate pointwise estimates of exponential sums to mean value type estimates using the following classical lemma.

\begin{llll}[Gallagher-Sobolev inequality]$\label{lem4.2}$

Let $f\ :\ [a,b]\rightarrow\C$ be continuously differentiable. Then

$$|f(u)|\leq (b-a)^{-1}\dint_a^b|f(x)|dx+\dint_a^b|f'(x)|dx $$
for any $u\in[a,b].$
\end{llll}

\bigskip
In order to describe the following proposition, we define the sets $\mathcal{D}_1=\mathcal{D}_1(\alpha)$ and $\mathcal{D}_2=\mathcal{D}_2(\alpha)$ with $\alpha\in\R$ by 
\begin{equation*}
    \mathcal{D}_1=\{h\in [1,H]\cap \Z|\ h\alpha\in \mathfrak{M}\ \text{mod}\ 1\}
\end{equation*}
and 
\begin{equation*}
     \mathcal{D}_2=\{h\in [1,H]\cap \Z|\ h\alpha\in \mathfrak{m}\ \text{mod}\ 1\}.
\end{equation*}

\begin{pr}
 Let $\alpha\in \R$, and $H>0.$ Suppose that $|q\alpha-a|\leq q^{-1}$ with $(q,a)=1.$ %Suppose that $\mathcal{D}_1$ and $\mathcal{D}_2$ are  subsets of $[1,H]\cap \Z$ such that for all $h\in \mathcal{D}_1\ (\text{or}\ \mathcal{D}_2)$, we have $h\alpha\in \mathfrak{M}\ (\text{or}\ \mathfrak{m})$, respectively. 
 Then, we have
 \begin{equation}\label{4.34.3}
 \dsum_{h\in \mathcal{D}_1}\bigl|\dsum_{1\leq x\leq X}e(h\alpha x^k)\bigr|^{k+1}\ll  H\left(q^{-1}+H^{-1}+qH^{-1}X^{-k}\right) X^{k+1+\epsilon},
 \end{equation}
 and
 \begin{equation}\label{4.44.4}
\dsum_{h\in \mathcal{D}_2}\bigl|\dsum_{1\leq x\leq X}e(h\alpha x^k)\bigr|^{k(k+1)}\ll  H\left(q^{-1}+H^{-1}+qH^{-1}X^{-k}\right)X^{k(k+1)-1+\epsilon}.    
 \end{equation}
\end{pr}

By applying Lemma $\ref{lem4.2}$, we shall derive upper bounds for the left hand side in $(\ref{4.34.3})$ and $(\ref{4.44.4})$ in terms of mean values of exponential sums
\begin{equation*}
    \dint_{\mathfrak{M}}\biggl|\dsum_{1\leq x\leq X}e(\alpha x^k)\biggr|^{k+1}d\alpha 
\end{equation*}
and 
\begin{equation*}
    \dint_{\mathfrak{m}}\biggl|\dsum_{1\leq x\leq X}e(\alpha x^k)\biggr|^{k(k+1)}d\alpha.
\end{equation*}
It follows from [$\ref{ref10}$, Theorem 4.4] and [$\ref{ref12}$, Theorem 2.1] that we shall obtain upper bounds for these mean values, and thus we complete the proof of Proposition 4.4. We emphasize here that the choice of exponents $k+1$ and $k(k+1)$ delivers the efficient application of [$\ref{ref10}$, Theorem 4.4] and [$\ref{ref12}$, Theorem 2.1].

\begin{proof}[Proof of Proposition 4.4]

We shall first derive ($\ref{4.34.3}$). Define a set $\Gamma(h)$ to be 
$$\Gamma(h)=\{\gamma\in [0,1)|\ \|h\alpha-\gamma\|<(4k)^{-1}X^{-k}\}.$$
By applying Lemma $\ref{lem4.2}$ to $\sum_{1\leq x\leq X}e(h\alpha x^k)$,
 one has
\begin{equation}\label{37}
    \begin{aligned}
    &\dsum_{h\in \mathcal{D}_1}\bigl|\dsum_{1\leq x\leq X}e(h\alpha x^k)\bigr|^{k+1} \\
    &\ll \dsum_{h\in \mathcal{D}_1}\left(X^k\dint_{\Gamma(h)}\bigl|\dsum_{1\leq x\leq X}e(\gamma x^k)\bigr|d\gamma+\dint_{\Gamma(h)}\bigl|\dsum_{1\leq x\leq X}x^ke(\gamma x^k)\bigr|d\gamma\right)^{k+1}\\
    &\ll  \dsum_{h\in \mathcal{D}_1}\left(X^k\dint_{\Gamma(h)}\bigl|\dsum_{1\leq x\leq X}e(\gamma x^k)\bigr|d\gamma\right)^{k+1}+\dsum_{h\in \mathcal{D}_1}\left(\dint_{\Gamma(h)}\bigl|\dsum_{1\leq x\leq X}x^ke(\gamma x^k)\bigr|d\gamma\right)^{k+1},
    \end{aligned}
\end{equation}
where we used $(A+B)^{k+1}\ll A^{k+1}+B^{k+1}$ for the second inequality. For concision, we write $\Xi_1$ and $\Xi_2$ for the first term and the second term in the bound ($\ref{37}$). Furthermore, for the sake of the next discussion, we freely assume that $X$ is an integer.

We first analyse the sum $\Xi_2.$
%$$\dsum_{h\in H_1}\left|\dint_{\|h\alpha-\gamma\|<\frac{1}{X^k}}\left|\dsum_{1\leq n\leq X}n^ke(\gamma n^k)\right|d\gamma\right|^{k+1},$$
By applying partial summation, we have
\begin{equation*}
\begin{aligned}
    &\dsum_{1\leq x\leq X}x^k e(\gamma x^k)=X^kS_{X+1}-S_{1}- \dsum_{2\leq x \leq X}(x^k-(x-1)^k)S_x,
\end{aligned}
\end{equation*}
where $$S_x=\dsum_{x\leq m\leq 2X}e(\gamma m^k).$$ Then, we find that $\Xi_2$ is
\begin{equation}\label{38}
\begin{aligned}
&\ll \dsum_{h\in \mathcal{D}_1}\biggl(\left(X^k\dint_{\Gamma(h)}|S_{X+1}|d\gamma\right)^{k+1}+\left(X^{k-1}\dsum_{2\leq x\leq X}\dint_{\Gamma(h)}|S_x|d\gamma\right)^{k+1}+\left(\dint_{\Gamma(h)}|S_{1}|d\gamma\right)^{k+1}\biggr).
\end{aligned}
\end{equation}
Meanwhile, on noting that mes$(\Gamma(h))\asymp X^{-k}$ and by applying H\"older's inequality, we have
\begin{equation*}
    \left(\dint_{\Gamma(h)}|S_x|d\gamma\right)^{k+1}\leq X^{-k^2}\dint_{\Gamma(h)}|S_x|^{k+1}d\gamma.
\end{equation*}
Thus, we deduce from ($\ref{38}$) that 
\begin{equation}\label{39}
\begin{aligned}
   & \Xi_2\ll X^k \sup_{1\leq x\leq X+1}\dsum_{h\in \mathcal{D}_1}\dint_{\Gamma(h)}|S_x|^{k+1}d\gamma.
 \end{aligned}
 \end{equation}

 Note that if $h\alpha\in\mathfrak{M}$, there exists $q\in\N$ with $1\leq q\leq X$ such that $\|qh\alpha\|\leq (2k)^{-1}X^{1-k}$. Thus, when $\|h\alpha-\gamma\|\leq (4k)^{-1}X^{-k}$ and $h\alpha \in \mathfrak{M}$, one has $\left\|q\gamma\right\|\leq \|qh\alpha\|+\|q(h\alpha-\gamma)\|\leq (2k)^{-1}X^{1-k}+(4k)^{-1}qX^{-k}\leq k^{-1}X^{1-k}$. Thus, on recalling the definition ($\ref{eq1.6}$) of $\mathfrak{M}_l$, one finds that $h\alpha\in \mathfrak{M}$ and $\|h\alpha-\gamma\|<(4k)^{-1}X^{-k}$ implies $\gamma\in \mathfrak{M}_1$. Let us write $$M(H,\gamma)=|\{h\in [1,H]\cap \Z|\ \|h\alpha-\gamma\|<(4k)^{-1}X^{-k}\}|$$ and $$M(H)=\displaystyle\sup_{\gamma\in [0,1)}M(H,\gamma).$$ Hence, by discussion above, we infer from ($\ref{39}$) that 
\begin{equation}\label{40}
\Xi_2\ll X^kM(H)\sup_{1\leq x\leq X+1}\dint_{\mathfrak{M}_1}|S_x|^{k+1}d\gamma.
   \end{equation}
 Meanwhile, by applying [$\ref{ref8}$, Lemma 6], one has
$$M(H)\ll H\left(q^{-1}+H^{-1}+qH^{-1}X^{-k}\right)$$
Furthermore, the Hardy-Littlewood method [$\ref{ref10}$, Theorem 4.4] readily confirms that 
$$\dint_{\mathfrak{M}_1}|S_x|^{k+1}d\gamma\ll X^{1+\epsilon}.$$ Therefore, we see from ($\ref{40}$) that
\begin{equation}\label{41}
\begin{aligned}
&\Xi_2\ll H\left(q^{-1}+H^{-1}+qH^{-1}X^{-k}\right)X^{k+1+\epsilon}.  
\end{aligned}
\end{equation} 

Next, it remains to estimate $\Xi_1$. By applying H\"older's inequality, we deduce that
\begin{equation}\label{42}
\begin{aligned}
   \Xi_1&\ll X^k\dint_{\Gamma(h)}|S_1-S_{X+1}|^{k+1}d\gamma\ll X^k \sup_{1\leq x\leq X+1}\dsum_{h\in \mathcal{D}_1}\dint_{\Gamma(h)} |S_x|^{k+1}d\gamma.
    \end{aligned}
\end{equation}
Then, by the same argument from ($\ref{39}$) to ($\ref{41}$), we have 
\begin{equation}\label{4222}
    \Xi_1\ll H(q^{-1}+H^{-1}+qH^{-1}X^{-k})X^{k+1+\epsilon}.
\end{equation}
Therefore, by $(\ref{37})$, $(\ref{41})$ and $(\ref{4222}),$ we conclude that
\begin{equation}\label{43}
    \dsum_{h\in \mathcal{D}_1}\biggl|\dsum_{1\leq x\leq X}e(h\alpha x^k)\biggr|^{k+1}\ll H\left(q^{-1}+H^{-1}+qH^{-1}X^{-k}\right)X^{k+1+\epsilon}.
\end{equation}
This confirms the estimate $(\ref{4.34.3}).$

We next derive ($\ref{4.44.4}$). Recall the definition $(\ref{eq1.6})$ of $\mathfrak{M}_l$ and $\mathfrak{m}_l=[0,1)\setminus \mathfrak{M}_l$. Note that if $h\alpha\in \mathfrak{m}$ and $\|h\alpha-\gamma\|<(4k)^{-1}X^{-k}$, then $\gamma\in\mathfrak{m}_4$. Indeed, if $\gamma\in \mathfrak{M}_4$, there exists $q\in \N$ with $1\leq q\leq X$ such that $\|q\gamma\|\leq (4k)^{-1}X^{1-k}$, and thus one has
$\|qh\alpha\|\leq \|q(h\alpha-\gamma)\|+\|q\gamma\|\leq q(4k)^{-1}X^{-k}+(4k)^{-1}X^{1-k} \leq (2k)^{-1}X^{1-k} ,$ which contradicts $h\alpha\in \mathfrak{m}.$

Therefore, the same treatment leading from ($\ref{37}$) to ($\ref{42}$) with the exponent $k(k+1)$ in place of $k+1$ gives the upper bound
\begin{equation}\label{44}
    \dsum_{h\in \mathcal{D}_2}\biggl|\dsum_{1\leq x\leq X}e(h\alpha x^k)\biggr|^{k(k+1)}\ll  X^kH\left(q^{-1}+H^{-1}+qH^{-1}X^{-k}\right)\sup_{1\leq x\leq X+1}\dint_{\mathfrak{m}_4}|S_x|^{k(k+1)}d\gamma.
\end{equation}
An application of the argument used in [$\ref{ref12}$, Theorem 2.1] confirms that
\begin{align*}
&\dint_{\mathfrak{m}_4}|S_x|^{k(k+1)}d\gamma\ll X^{k(k+1)-k-1+\epsilon}.
\end{align*}
Thus, on substituting this estimate into ($\ref{44}$), we obtain $(\ref{4.44.4})$. Therefore, we complete the proof of Proposition 4.4.
\end{proof}

\begin{rmk}
Recall from section 4.2 that under the assumption ($\ref{4.14.1}$), we may assume that $\alpha_j\in\mathfrak{m}^H$ with $1\leq j\leq s$. For a given index $j$ with $1\leq j\leq s$, it follows Dirichilet's approximation theorem that there exists $a\in \Z$ and $q\in \N$ with $1\leq q\leq HX^{k-1}$ and $(q,a)=1$ such that $|q\alpha_j-a|\leq H^{-1}X^{1-k}.$ Since $\alpha_j\in \mathfrak{m}^H$, moreover, one has $q>X.$ Thus, Proposition 4.4 with the assumption ($\ref{4.14.1}$) delivers that %, whenever $|\alpha-a/q|\leq q^{-1}X^{-k+1}H^{-1}$, one has $X<q\leq X^{k-1}H$. Note that $X< q< X^{k-1}H$ implies 
%$$H\left(q^{-1}+H^{-1}+qH^{-1}X^{-k}\right)\leq 1.$$
for $1\leq j\leq s$ one has
\begin{equation}\label{45}
    \dsum_{h\in \mathcal{D}_1(\alpha_j)}\biggl|\dsum_{1\leq x\leq X}e(h\alpha_j x^k)\biggr|^{k+1}\ll (1+H/X)X^{k+1+\epsilon},
    \end{equation}
    and
    \begin{equation}\label{4646}
    \dsum_{h\in \mathcal{D}_2(\alpha_j)}\biggl|\dsum_{1\leq x\leq X}e(h\alpha_j x^k)\biggr|^{k(k+1)}\ll (1+H/X)X^{k(k+1)-1+\epsilon}.        
    \end{equation}
\end{rmk}

\bigskip

\subsection{Proof of Theorem 4.1}

\begin{proof}
Let $H=X^{\sigma(s,k)-\nu}$ for sufficiently small $\nu>0.$ Suppose that
\begin{equation}\label{ineq4.18}
    \min_{\substack{0\leq \boldsymbol{x}\leq X\\ \boldsymbol{x}\neq \boldsymbol{0}}}\|\alpha_1x_1^k+\alpha_2x_2^k+\cdots+\alpha_sx_s^k\| > H^{-1}.
\end{equation}
From section 4.1, recall that the sets $H_1,\ldots, H_{2^s}$ are such that the set
$\{h\alpha_j|\ h\in H_i\}\subseteq\mathfrak{M}$ or $\{h\alpha_j|\ h\in H_i\}\subseteq\mathfrak{m},$ for all $1\leq j\leq s$ and $1\leq i\leq 2^s.$
By relabelling $\alpha_i$, we may assume that for $1\leq i\leq m$, the set $\{h\alpha_i|\ h\in H_1\}\subseteq \mathfrak{M}$, and for $m+1\leq i\leq s$, the set $\{h\alpha_i|\ h\in H_1\}\subseteq \mathfrak{m}.$ Note from the explanation following the proof of Proposition 4.4 that we have $(\ref{45})$ and $(\ref{4646})$.

We first consider the case when $m\geq k+1.$ Recall from section 4.2 that the assumption ($\ref{ineq4.18}$) implies that $\alpha_j\in \mathfrak{m}^H$ with $1\leq j\leq s.$ Then, by making use of our hypothesis $s\geq k+2$, together with H\"older's inequality and ($\ref{46}$), we deduce that
\begin{equation}\label{4.184.18}
\begin{aligned}
 &   \dsum_{h\in H_1}\biggl|\dsum_{1\leq \boldsymbol{x}\leq X}e(h(\alpha_1x_1^k+\alpha_2x_2^k+\cdots+\alpha_sx_s^k))\biggr|\\
  & \ll X^{s-(k+1)-\delta_1}\dprod_{1\leq j\leq k+1}\biggl(\dsum_{h\in H_1}\biggl|\dsum_{1\leq x_j\leq X}e(h\alpha_jx_j^k)\biggr|^{k+1}\biggr)^{\frac{1}{k+1}}.
\end{aligned}
\end{equation}
Meanwhile, on recalling the definition of $H_1$ and $\mathcal{D}_1$ following Lemma 4.3, we notice that $H_1\subseteq \mathcal{D}_1(\alpha_j)$ for $1\leq j\leq k+1.$ Then, by applying ($\ref{45}$) with $H\leq X$, it follows from ($\ref{4.184.18}$) that
\begin{equation*}
\begin{aligned}
 &   \dsum_{h\in H_1}\biggl|\dsum_{1\leq \boldsymbol{x}\leq X}e(h(\alpha_1x_1^k+\alpha_2x_2^k+\cdots+\alpha_sx_s^k))\biggr|\\
  &\ll X^{s-(k+1)-\delta_1}\dprod_{1\leq j\leq k+1}\biggl(\dsum_{h\in \mathcal{D}_1(\alpha_j)}\biggl|\dsum_{1\leq x_1\leq X}e(h\alpha_jx_j^k)\biggr|^{k+1}\biggr)^{\frac{1}{k+1}}\ll X^{s-\eta},
\end{aligned}
\end{equation*}
for some $\eta=\eta(\delta_1)>0.$

Next, consider the case when $m<k+1.$ We write 
\begin{equation*}
    A_i=\dsum_{h\in H_1}\bigl|\dsum_{1\leq x_i\leq X}e(h\alpha_i x_i^k)\bigr|^{k+1},\ B_i=\dsum_{h\in H_1}\bigl|\dsum_{1\leq x_i\leq X}e(h\alpha_i x_i^k)\bigr|^{k(k+1)},
\end{equation*}
and put $m_1=\min\{k(k+1-m), s-m\}$. Then it follows from H\"older's inequality that
\begin{equation}\label{ineq48}
    \begin{aligned}
   & \dsum_{h\in H_1}\biggl|\dsum_{1\leq \boldsymbol{x}\leq X}e(h(\alpha_1x_1^k+\alpha_2x_2^k+\cdots+\alpha_sx_s^k))\biggr|\\
    &\ll \left(\dsum_{h\in H_1}1\right)^{1-\frac{km+m_1}{k(k+1)}}A_1^{\frac{1}{k+1}}\cdots A_m^{\frac{1}{k+1}}B_{m+1}^{\frac{1}{k(k+1)}}\cdots B_{m+m_1}^{\frac{1}{k(k+1)}}X^{s-(m+m_1)}.
    \end{aligned}
\end{equation}

 On recalling the definition $H_1$, $\mathcal{D}_1$ and $\mathcal{D}_2$ following Lemma 4.3, notice that $H_1\subseteq \mathcal{D}_1({\alpha_i})$ for $1\leq i\leq m$, and $H_1\subseteq \mathcal{D}_2({\alpha_i})$ for $m+1\leq i\leq m+m_1.$ Thus, for $1\leq i\leq m$ we have 
$$A_i\leq \dsum_{h\in \mathcal{D}_1(\alpha_i)}\bigl|\dsum_{1\leq x_i\leq X}e(h\alpha_i x_i^k)\bigr|^{k+1}$$
and for $m+1\leq i\leq m+m_1$ we have
$$B_i\leq \dsum_{h\in \mathcal{D}_2(\alpha_i)}\bigl|\dsum_{1\leq x_i\leq X}e(h\alpha_i x_i^k)\bigr|^{k(k+1)}.$$
Then, on substituting these inequalities into ($\ref{ineq48}$), it follows by applying ($\ref{45}$) and ($\ref{4646}$) that
\begin{equation}\label{4.194.19}
\begin{aligned}
& \dsum_{h\in H_1}\biggl|\dsum_{1\leq \boldsymbol{x}\leq X}e(h(\alpha_1x_1^k+\alpha_2x_2^k+\cdots+\alpha_sx_s^k))\biggr|\\
& \ll H^{1-\frac{km+m_1}{k(k+1)}}X^{m}X^{m_1-\frac{m_1}{k(k+1)}}X^{s-(m+m_1)}X^{\epsilon}.
      \end{aligned}
\end{equation}
Recall that $H=X^{\sigma(s,k)-\nu}.$ Then, the right hand side in ($\ref{4.194.19}$) is $O(X^{\phi})$ where 
\begin{equation}\label{4.22}
\phi=s+\left(1-\frac{km+m_1}{k(k+1)}\right)(\sigma(s,k)-\nu)-\frac{m_1}{k(k+1)}+\epsilon.    
\end{equation}%Note here that two $\epsilon$'s in $\phi$ may be chosen separately, but we use the same denotation to adopt the convention of $\epsilon$ as we explained at the end of section 1.

We shall show that $\phi\leq s-\eta$ for some $\eta>0$. Recall the definition of $m_1$. When $m\geq\frac{k(k+1)-s}{k-1}$, one has $m_1=k(k+1-m)$. Thus, one has $\phi= s-1+\frac{m}{k+1}+\epsilon< s-\eta$ for some $\eta>0$, since $m< k+1$. When $m< \frac{k(k+1)-s}{k-1}$, one has $m_1=s-m$. In this case, we have $1-\frac{km+m_1}{k(k+1)}>0$, and thus it follows from ($\ref{4.22}$) that
\begin{equation}\label{eq4.22}
    \phi=s+\left(1-\frac{km+m_1}{k(k+1)}\right)\sigma(s,k)-\frac{m_1}{k(k+1)}-\eta,
\end{equation}
for some $\eta=\eta(\nu)$.

First, consider the case $s\geq k(k+1)/2$. Then, it follows from ($\ref{333}$) that $\sigma(s,k)=1$. Hence, since $m_1=s-m,$ it follows from ($\ref{eq4.22}$) that $$\phi= s+\left(1-\frac{(k-2)m+2s}{k(k+1)}\right)-\eta,$$ for some $\eta=\eta(\nu)>0.$ Hence, it follows by $s\geq k(k+1)/2$ and $m\geq 0$ that $\phi\leq s-\eta$ for some $\eta>0.$ Next, recall the hypothesis $s\geq k+2$ in the statement of Theorem 4.1, and consider next the case $k+2\leq s\leq k(k+1)/2.$ Then, it follows from $(\ref{333})$ that $\sigma(s,k)=\frac{s}{k(k+1)-s}.$ Hence, since $m_1=s-m$, it follows from $(\ref{eq4.22})$ that
\begin{equation}
\begin{aligned}
\phi&=s+\biggl(\frac{k(k+1)-s}{k(k+1)}+\frac{-km+m}{k(k+1)}\biggr)\left(\frac{s}{k(k+1)-s}\right)-\frac{s-m}{k(k+1)}-\eta\\
 &=s+\frac{s}{k(k+1)}+\left(\frac{(-km+m)s}{k(k+1)(k(k+1)-s)}\right)-\frac{s-m}{k(k+1)}-\eta\\
 &=s+\frac{m}{k(k+1)}\left(1-\frac{(k-1)s}{k(k+1)-s}\right)-\eta,
\end{aligned}   
\end{equation}
 for some $\eta=\eta(\nu)>0$. Hence, it follows by $s\geq k+2$ and $m\geq 0$ that $\phi\leq s-\eta$ for some $\eta>0.$ Therefore, in all cases, we have
\begin{equation}
\dsum_{h\in H_1}\biggl|\dsum_{1\leq \boldsymbol{x}\leq X}e(h(\alpha_1x_1^k+\alpha_2x_2^k+\cdots+\alpha_sx_s^k))\biggr|\ll X^{s-\eta},
\end{equation}
for some $\eta>0.$ Then, by the same treatment, we have ($\ref{35}$) for every $H_i\ (i=1,\ldots,2^s)$, which contradicts ($\ref{34}$) stemming from ($\ref{ineq4.18}$). Therefore, we are forced to conclude that 
\begin{equation*}
    \min_{\substack{0\leq\boldsymbol{x}\leq X\\ \boldsymbol{x}\neq \boldsymbol{0}}}\|\alpha_1x_1^k+\alpha_2x_2^k+\cdots+\alpha_sx_s^k\| \leq H^{-1}.
\end{equation*}
Hence, by letting $\nu\rightarrow 0$, we complete the proof of Theorem 4.1.
\end{proof}

\section{Proof of Theorem 1.2}

\bigskip

In this section, we provide the proof of Theorem 1.2. We recall the major arcs $\mathfrak{M}=\mathfrak{M}_2$ defined in ($\ref{eq1.6}$), and their complement $\mathfrak{m}=\mathfrak{m}_2.$ In the proof of Theorem 4.1, we used major arcs estimates [$\ref{ref10}$, Theorem 4.4]
\begin{equation}\label{51}
    \dint_{\mathfrak{M}}\biggl|\dsum_{1\leq x\leq X}e(\alpha x^k)\biggr|^{k+1}d\alpha \ll X^{1+\epsilon}
\end{equation}
and minor arcs estimates [$\ref{ref12}$, Theorem 2.1]
\begin{equation}\label{52}
    \dint_{\mathfrak{m}}\biggl|\dsum_{1\leq x\leq X}e(\alpha x^k)\biggr|^{k(k+1)}d\alpha \ll X^{k(k+1)-k-1+\epsilon}.
\end{equation}
To prove Theorem 1.2, we replace the mean values ($\ref{51}$) and ($\ref{52}$) with those in Theorem 1.4, and follow the same argument with the proof of Theorem 4.1.%We use the same argument with the above to prove Theorem 4. However, we replace the major and minor arcs estimations by the ones in Theorem 2.  

\subsection{Outline of the proof of Theorem 1.2}

Let $s>k_1^2+k_1+2\lceil\sigma(1-k_1)\rceil$. Throughout this section, we put $H=X^{1-\nu}$ for sufficiently small $\nu>0$ unless specified otherwise. Recall $\varphi_j(x)=\alpha_{1j}x^{k_1}+\cdots+\alpha_{tj}x^{k_t}$. Suppose that 
\begin{equation}\label{ineq5.35.3}
    \min_{\substack{0\leq \boldsymbol{x}\leq X\\\boldsymbol{x}\neq \boldsymbol{0}}}\|\varphi_1(x_1)+\varphi_2(x_2)+\cdots +\varphi_{s}(x_{s})\|> H^{-1}.
\end{equation} Then, by Lemma 4.2, we have
\begin{equation}\label{53}
    \dsum_{1\leq h\leq H}\bigl|\dsum_{1\leq \boldsymbol{x}\leq X}e(h(\varphi_1(x_1)+\cdots+\varphi_{s}(x_{s})))\bigr|\gg X^{s}.
\end{equation}

 On observing that each real number $h\alpha_{1j}$ lies either on $\mathfrak{M}$ or $\mathfrak{m}$, one can decompose the set $[1,H]\cap \Z$ into $2^s$ sets, $H_1,\ldots,H_{2^{s}}$, such that the set $\{h\alpha_{1j}|\ h\in H_i\}\subseteq \mathfrak{M}$ or $\{h\alpha_{1j}|\ h\in H_i\}\subseteq \mathfrak{m}$, for all $1\leq j\leq s$ and $1\leq i\leq 2^s.$ Our goal is to show that for every $H_i\ (i=1,\ldots,2^s)$, we have
\begin{equation}\label{54}
    \dsum_{h\in H_i}\bigl|\dsum_{1\leq \boldsymbol{x}\leq X}e(h(\varphi_1(x_1)+\cdots+\varphi_s(x_{s})))\bigr|\ll X^{s-\eta},
\end{equation}
for some $\eta=\eta(k,\nu)>0$. This contradicts ($\ref{53}$) for sufficiently large $X$ in terms of $\eta$ and $s$. Thus, this forces us to conclude that whenever $s>k_1^2+k_1+2\lceil\sigma(1-k_1)\rceil$ and $X$ is sufficiently large, one has
\begin{equation*}
    \min_{\substack{0\leq \boldsymbol{x}\leq X\\ \boldsymbol{x}\neq \boldsymbol{0}}}\|\varphi_1(x_1)+\varphi_2(x_2)+\cdots +\varphi_{s}(x_{s})\|\leq H^{-1}.
\end{equation*} 
Therefore, by letting $\nu\rightarrow 0$, we are done to prove Theorem 1.2.

\bigskip

\subsection{Preliminary manoeuvre}
As in the previous section, we can obtain extra information about $\alpha_{ij}$ with $1\leq i\leq t, 1\leq j\leq s$, under the assumption ($\ref{ineq5.35.3}$). In order to describe this information, we must define $\widetilde{\mathfrak{M}}_H$ by

\begin{equation*}
    \widetilde{\mathfrak{M}}_H=\bigcup_{\substack{0\leq a_1,\ldots,a_t\leq q\leq X\\(q,a_1,\ldots,a_t)=1}} \widetilde{\mathfrak{M}}_H(q,a_1,\ldots,a_t),
\end{equation*}
where $$\widetilde{\mathfrak{M}}_H(q,a_1,\ldots,a_t)=\{(\alpha_1,\ldots,\alpha_t)\in [0,1)^t|\ |\alpha_i-a_i/q|\leq t^{-1}q^{-1}X^{-k_i+1}H^{-1}\ \text{for}\ 1\leq i\leq t\}.$$ Define $\widetilde{\mathfrak{m}}_H=[0,1)\setminus \widetilde{\mathfrak{M}}_H$.
Note that if there exists $j$ such that $(\alpha_{1j},\ldots,\alpha_{tj})\in \widetilde{\mathfrak{M}}_H$, it follows by putting $x_j=q$ and $x_i=0\ (i\neq j)$ that
\begin{multline*}
    \min_{\substack{0\leq \boldsymbol{x}\leq X\\ \boldsymbol{x}\neq \boldsymbol{0}}}\|\varphi_1(x_1)+\cdots+\varphi_{s}(x_{s})\|\leq \min_{1\leq x_j\leq X}\|\varphi_j(x_j)\|\leq \|\varphi_j(q)\|\\
    \leq q^{k_1-1}\|q\alpha_{1j}\|+q^{k_2-1}\|q\alpha_{2j}\|+\cdots+q^{k_t-1}\|q\alpha_{tj}\|\leq H^{-1},
\end{multline*}
which contradicts ($\ref{ineq5.35.3}$). Hence, under the assumption ($\ref{ineq5.35.3}$), we may assume that $(\alpha_{1j},\ldots,\alpha_{tj})$ is in $\widetilde{\mathfrak{m}}_H$ for every $j=1,\ldots,s.$

%Note that since $(\alpha_{1j},\alpha_{2j},...,\alpha_{tj})\in \tilde{\mathfrak{m}}_H,$ there exists $q,a_1,...,a_t$ with  $(q,a_1,...,a_t)=1$ such that $X<q<X^{i-1}H$ and
%\begin{equation}
%    |\alpha_{ij}-\frac{a_i}{q}|<\frac{1}{qX^{i-1}H}
%\end{equation}
%for all $i(i=1,...,t).$
Furthermore, whenever $(\alpha_{1j},\alpha_{2j},\ldots,\alpha_{tj})\in \widetilde{\mathfrak{m}}_H$ with $H\leq X^{1-\nu}$ for sufficiently small $\nu>0$, one has for all $h\in [1,H]\cap \Z$
\begin{equation}\label{62}
    \dsum_{1\leq x\leq X}e(h(\alpha_{1j}x^{k_1}+\cdots+\alpha_{tj}x^{k_t}))\ll X^{1-\delta_1}
\end{equation}
for some positive number $\delta_1=\delta_1(k_1,\nu)$. Indeed, suppose that there exists $h\in H$ such that 
\begin{equation*}
    \dsum_{1\leq x\leq X}e(h(\alpha_{1j}x^{k_1}+\cdots+\alpha_{tj}x^{k_t}))\geq X^{1-\delta_1}.
\end{equation*}
Then, by [$\ref{ref2}$, Theorem 4.3] and [$\ref{ref2}$, Lemma 4.6], there exist $q,a_1,\ldots,a_t$ such that $q<X^{\eta}$ and
$$|h\alpha_{ij}-a_i/q|<q^{-1}X^{\eta-k_i}\ (i=1,\ldots,t)$$
where $\eta=\eta(\delta_1,k_1).$ This gives
$$|\alpha_{ij}-a_i/(qh)|<(qh)^{-1}X^{\eta-k_i}\ (i=1,\ldots,t).$$
For sufficiently small $\delta_1$ so that $\eta$ is smaller than $\nu,$ one has $qh<X^{\eta}X^{1-\nu}<X$ and $$|\alpha_{ij}-a_i/(qh)|<(qh)^{-1}X^{1-k_i}H^{-1}\ (i=1,\ldots,t).$$
By dividing the greatest common divisor of $a_i$ and $qh$, this readily confirms that ($\alpha_{1j},\ldots,\alpha_{tj}$) $\in \widetilde{\mathfrak{M}}_H,$ which contradicts  ($\alpha_{1j},\ldots,\alpha_{tj}$) $\in \widetilde{\mathfrak{m}}_H. $
%we shall show that whenever $(\alpha_{1j},\alpha_{2j},...,\alpha_{tj})\in \tilde{\mathfrak{m}}_H$, one has for all $h\in H_1$

%Suppose that there exist $h\in H_1$ such that 
%\begin{equation*}
 %   \dsum_{1\leq x \leq X}e(h(\alpha_{1j}x_j^{k_1}+\cdots+\alpha_{tj}x_j^{k_t}))\gg X^{1-\epsilon_1}.
%\end{equation*}
%Then, by [$\ref{ref2}$, Theorem 4.3] and [$\ref{ref2}$, Lemma 4.6], there exists $q,a_1,....,a_t$ such that $q< X^{\delta}$ and
%\begin{equation*}
 %   |h\alpha_{ij}-\frac{a_i}{q}|< \frac{X^{\delta}}{qX^{k_i}}\ (i=1,...,t)
%\end{equation*}
%where $\delta=\delta(\epsilon_1)$. This gives 
%\begin{equation*}
 %   |\alpha_{ij}-\frac{a_i}{qh}|<\frac{X^{\delta}}{qhX^{k_i}}\ (i=1,...,t).
%\end{equation*}
 %By taking sufficienly small $\delta=\delta(\epsilon_1)$, one has $qh<X^{\delta}X^{1-\epsilon}<X$ and
%\begin{equation*}
 %    |\alpha_{ij}-\frac{a_i}{qh}|<\frac{1}{qhX^{k_i-1}H}\ (i=1,...,t).
%\end{equation*}
%This yields that $(\alpha_{1j},\alpha_{2j},...,\alpha_{tj})\in \tilde{\mathfrak{M}}_H,$ which contradicts $(\alpha_{1j},\alpha_{2j},...,\alpha_{tj})\in \tilde{\mathfrak{m}}_H.$

\bigskip

\subsection{Auxiliary proposition}

Recall the definition ($\ref{eq1.41.4}$) of $\sigma$ with $\mathbf{k}=(k_1,\ldots,k_t)$. To show ($\ref{54}$), we require following proposition analogous to Proposition 4.4. In order to describe the following proposition, it is convenient to define $N(H,\boldsymbol{\gamma},\alpha_1,\ldots,\alpha_t)$ with $\boldsymbol{\gamma}\in [0,1)^t$, $(\alpha_1,\ldots,\alpha_t)\in [0,1)^t$ and $H>0$ by 
\begin{equation*}
    N(H,\boldsymbol{\gamma},\alpha_1,\ldots,\alpha_t)=|\{h\in [1,H]\cap \Z|\ \|h\alpha_j-\gamma_j\|<(4k)^{-1}X^{-k_j}\ \textrm{for}\ j=1,\ldots,t\}|,
\end{equation*}
and define $N(H):=N(H,\alpha_1,\ldots,\alpha_t)=\sup_{\boldsymbol{\gamma}\in [0,1)^t}N(H,\boldsymbol{\gamma},\alpha_1,\ldots,\alpha_t).$ We recall the definition $\mathcal{D}_1=\mathcal{D}_1(\alpha)$ and $\mathcal{D}_2=\mathcal{D}_2(\alpha)$ with $\alpha\in \R$, following Lemma 4.3. Furthermore, let us put $L=(k_1^2+k_1)/2+\lceil\sigma(1-k_1)\rceil$.

\begin{pr}
Let $H>0.$ Suppose that $\alpha_j\in \R$ with $t\geq 2$ and $1\leq j\leq t$. Then, we have
\begin{equation}\label{5.75.7}
    \dsum_{h\in \mathcal{D}_1(\alpha_1)}\biggl|\dsum_{1\leq x\leq X}e(h(\alpha_1x^{k_1}+\alpha_2x^{k_2}+\cdots+\alpha_t x^{k_t}))\biggr|^{2L}\ll N(H)X^{2L+\epsilon},
\end{equation}
and
\begin{equation}\label{5.85.8}
        \dsum_{h\in \mathcal{D}_2(\alpha_1)}\biggl|\dsum_{1\leq x\leq X}e(h(\alpha_1x^{k_1}+\alpha_2x^{k_2}+\cdots+\alpha_t x^{k_t}))\biggr|^{k_1(k_1+1)}\ll N(H)X^{k_1(k_1+1)-\sigma+\epsilon}.
\end{equation}
\end{pr}
We shall first derive upper bounds for $(\ref{5.75.7})$ and $(\ref{5.85.8})$ in terms of the left hand side of $(\ref{7777})$ with $2L$ in place of $2s$, and $(\ref{8888})$ with $k_1(k_1+1)$ in place of $2s$. Then, by applying Theorem $\ref{thm1.2}$, we complete the proof of Proposition 5.1. We note here that the choice of $2L$ and $k_1(k_1+1)$ delivers the efficient application of Theorem 1.4.

\bigskip

\begin{proof}[Proof of Proposition 5.1]
For simplicity, throughout this proof, we write $\mathcal{D}_1=\mathcal{D}_1(\alpha_1)$ and $\mathcal{D}_2=\mathcal{D}_2(\alpha_1).$ 
Define $\Gamma(h)$ to be 
$$\Gamma(h)=\{(\gamma_1,\ldots,\gamma_t)\in [0,1)^t|\ \|h\alpha_j-\gamma_j\|\leq (4k)^{-1}X^{-k}\}.$$
Recall the definition ($\ref{1.7}$) of $D$. By applying [$\ref{ref5}$, Lemma 1] to $$\sum_{1\leq x\leq X}e(h(\alpha_1x^{k_1}+\cdots+\alpha_tx^{k_t})),$$ we infer that
\begin{equation*}
 \dsum_{h\in \mathcal{D}_1}\biggl|\dsum_{1\leq x\leq X}e(h(\alpha_1x^{k_1}+\alpha_2x^{k_2}+\cdots+\alpha_t x^{k_t}))\biggr|^{2L}
\end{equation*}
\begin{equation}\label{58'}
     \ll X^D \dsum_{h\in \mathcal{D}_1}\dint_{\Gamma(h)}\biggl|\sup_{I\subseteq [1,X]}\dsum_{x\in I}e(\gamma_1x^{k_1}+\gamma_2x^{k_2}+\cdots+\gamma_t x^{k_t}))\biggr|^{2L} d\mbd{\gamma},%\\
    % \ll X^D \dint_{U_1}\dint_{U_2}\cdots\dint_{U_t}N(H_1) \left|\dsum_{1\leq x\leq X}e(\gamma_1x^{k_1}+\gamma_2x^{k_2}+\cdots+\gamma_t x^{k_t}))\right|^{2L}d\mbd{\gamma},
\end{equation}
where $I$ runs over all intervals in $[1,X].$
 In the proof of Proposition 4.4, we have seen that for $h\alpha_1\in \mathfrak{M}$, the set $\{\gamma_1|\ \|h\alpha_1-\gamma_1\|<(4k)^{-1}X^{-k_1}\}$ is a subset of $\mathfrak{M}_1$.  Then, by making use of $N(H)$, we deduce that the bound ($\ref{58'}$) is
\begin{equation}\label{81}
    \ll N(H)X^D \dint_{\mathfrak{M}_1}\dint_0^1\cdots\dint_0^1\biggl|\sup_{I\subseteq [1,X]}\dsum_{x\in I}e(\gamma_1x^{k_1}+\gamma_2x^{k_2}+\cdots+\gamma_t x^{k_t}))\biggr|^{2L}d\mbd{\gamma}.
\end{equation}
 Therefore, by applying the Caleson-Hunt theorem with respect to the integral over $\gamma_t$ and Theorem 1.4 ($i$) with $\mathfrak{M}=\mathfrak{M}_1$, one concludes that the bound ($\ref{81}$) is $O( N(H_1)X^{2L+\epsilon}).$ This confirms ($\ref{5.75.7}$).

Similarly, in the proof of Proposition 4.4, we have seen that for $h\alpha_1\in \mathfrak{m}$, the set $$\{\gamma_1|\ \|h\alpha_1-\gamma_1\|<X^{-k_1}\}$$ is a subset of $ \mathfrak{m}_4$. Thus, we infer that
\begin{align*}
 &\dsum_{h\in \mathcal{D}_2}\biggl|\dsum_{1\leq x\leq X}e(h(\alpha_1x^{k_1}+\alpha_2x^{k_2}+\cdots+\alpha_t x^{k_t}))\biggr|^{k_1(k_1+1)}\\
 & \ll N(H)X^D \dint_{\mathfrak{m}_4}\dint_0^1\cdots\dint_0^1\biggl|\sup_{I\subseteq [1,X]}\dsum_{x\in I}e(\gamma_1x^{k_1}+\gamma_2x^{k_2}+\cdots+\gamma_t x^{k_t}))\biggr|^{k_1(k_1+1)}d\mbd{\gamma},
\end{align*}
where $I$ runs over all intervals in $[1,X].$
Thus, by applying the Carleson-Hunt theorem with respect to the integral over $\gamma_t$ and Theorem 1.4 ($ii$) with $\mathfrak{m}=\mathfrak{m}_4$, we find that the last expression is $O(N(H)X^{k_1(k_1+1)-\sigma+\epsilon}).$ This confirms ($\ref{5.85.8}$).
\end{proof}

\begin{rmk}
The Carleson-Hunt Theorem could be avoided at the cost of a factor $\log (6X)$ by standard use of a Dirichlet kernel argument (see, for example, [$\ref{ref30}$, Lemma 7.1])
\end{rmk}
\begin{rmk}

 Recall from section 5.2 that under the assumption ($\ref{ineq5.35.3}$), we may assume that $(\alpha_{1j},\ldots,\alpha_{tj})$ is in $\widetilde{\mathfrak{m}}_H$ for every $j\ (j=1,\ldots,s)$. We see that whenever $(\alpha_{1j},\ldots,\alpha_{tj})\in \widetilde{\mathfrak{m}}_H$, we have $N(H,\alpha_{1j},\ldots,\alpha_{tj})\leq 1.$ Indeed, if $N(H)>1$, there exists $h_1$, $h_2$ ($1\leq h_1,h_2\leq H$, $h_1\neq h_2$) and $\mbd{\gamma}=(\gamma_1,\ldots,\gamma_t)\in [0,1)^t$ such that 
\begin{equation*}
    \|h_1\alpha_{ij}-\gamma_i\|< X^{-k_i},\ \|h_2\alpha_{ij}-\gamma_i\|<X^{-k_i} \ \ \  (i=1,\ldots,t).
\end{equation*}
By triangle inequality,
\begin{equation}\label{58}
    \|(h_1-h_2)\alpha_{ij}\|\leq \|h_1\alpha_{ij}-\gamma_i\|+\|h_2\alpha_{ij}-\gamma_i\|< 2X^{-k_i}
\end{equation}
for all $i$ $ (1\leq i \leq t).$
Since $2X^{-k_i}<t^{-1}X^{-k_i+1}H^{-1}$ for sufficiently large $X$, it follows from ($\ref{58}$) that for every $i\ (1\leq i \leq t)$
\begin{equation}
     \|(h_1-h_2)\alpha_{ij}\|<t^{-1}X^{-k_i+1}H^{-1}.
\end{equation}
 Since $0<|h_1-h_2|< X,$ one has $(\alpha_{1j},\alpha_{2j},\ldots,\alpha_{tj})\in \widetilde{\mathfrak{M}}_H.$ This contradicts our assumption that $(\alpha_{1j},\alpha_{2j},\ldots,\alpha_{tj})\in \widetilde{\mathfrak{m}}_H.$ Hence, Proposition 5.1 with the assumption ($\ref{ineq5.35.3}$) delivers that
 for every $j\ (j=1,\ldots,s)$ one has
 \begin{equation}\label{60}
   \dsum_{h\in \mathcal{D}_1(\alpha_{1j})}\biggl|\dsum_{1\leq x\leq X}e(h(\alpha_{1j}x^{k_1}+\cdots+\alpha_{tj}x^{k_t}))\biggr|^{2L}\ll X^{2L+\epsilon},
 \end{equation}
and
\begin{equation}\label{61}
     \dsum_{h\in \mathcal{D}_2(\alpha_{1j})}\biggl|\dsum_{1\leq x\leq X}e(h(\alpha_{1j}x^{k_1}+\cdots+\alpha_{tj}x^{k_t}))\biggr|^{k_1(k_1+1)}\ll X^{k_1(k_1+1)-\sigma+\epsilon}.
\end{equation}  
\end{rmk}

\bigskip

\subsection{Proof of Theorem 1.2}
\begin{proof}[Proof of Theorem 1.2]
% It suffices to show that for every $i$
%\begin{equation}\label{63}
 %   \dsum_{1\leq h\leq H_i}|\dsum_{1\leq x_1,...,x_{s}\leq X}e(h(\phi_1(x_1)+\cdots+\phi_s(x_{s})))|\ll X^{s-\eta}\ \textrm{for some}\ \eta>0.
%\end{equation}
Suppose that ($\ref{ineq5.35.3}$) holds. From section 5.1, recall that the set 
$\{h\alpha_{1j}|\ h\in H_i\}\subseteq\mathfrak{M}$ or $\{h\alpha_{1j}|\ h\in H_i\}\subseteq\mathfrak{m}$, for all $1\leq j\leq s$ and $1\leq i\leq 2^s.$
By relabelling $\alpha_{1j}$, we may assume that for $1\leq i\leq m$, the set $\{h\alpha_{1i}|\ h\in H_1\}\subseteq \mathfrak{M}$, and for $m+1\leq i\leq s$, the set $\{h\alpha_{1i}|\ h\in H_1\}$ is a subset of $\mathfrak{m}.$ We put again $L=(k_1^2+k_1)/2+\lceil\sigma(1-k_1)\rceil$ and recall that $(\alpha_{1j},\ldots,\alpha_{tj})$ is in $\widetilde{\mathfrak{m}}_{H}$ for every $j=1,\ldots,s$. Note from Remark 2 above and section 5.2 that we have ($\ref{60}$), ($\ref{61}$) and ($\ref{62}$).

We first consider the case $m\geq 2L.$ By making use of our hypothesis $s>2L$ together with H\"older's inequality and ($\ref{62}$), we deduce that %($\ref{60}$) and , it follows from H\"older's inequality that 
\begin{equation}\label{ineq5.15}
\begin{aligned}
    &\dsum_{ h\in H_1}\bigl|\dsum_{1\leq \boldsymbol{x}\leq X}e(h(\varphi_1(x_1)+\cdots+\varphi_s(x_{s})))\bigr|\\
    &\ll X^{s-2L-\delta_1}\dprod_{l=1}^{2L}\biggl(\dsum_{h \in H_1}\bigl|\dsum_{1\leq x_l\leq X}e(h\varphi_l(x_l))\bigr|^{2L}\biggr)^{1/2L}.
    \end{aligned}
\end{equation}
Meanwhile, on recalling the definition of $H_1$ and $\mathcal{D}_1$, we notice that $H_1\subseteq \mathcal{D}_1(\alpha_{1l})$ for $1\leq l\leq 2L$. Then, by applying ($\ref{60}$), it follows from $(\ref{ineq5.15})$ that 
\begin{equation*}
\begin{aligned}
    &\dsum_{ h\in H_1}\bigl|\dsum_{1\leq \boldsymbol{x}\leq X}e(h(\varphi_1(x_1)+\cdots+\varphi_s(x_{s})))\bigr|\\
    &\ll X^{s-2L-\delta_1}\dprod_{l=1}^{2L}\biggl(\dsum_{h \in \mathcal{D}_1(\alpha_{1l})}\bigl|\dsum_{1\leq x_l\leq X}e(h\varphi_l(x_l))\bigr|^{2L}\biggr)^{1/2L}\ll X^{s-\eta},
    \end{aligned}
\end{equation*}
for some $\eta=\eta(\delta_1)>0.$

 Next, consider the case $m<2L.$ We write 
\begin{equation*}
    \begin{aligned}
  &  A_l=\dsum_{h\in H_1}\bigl|\dsum_{1\leq x_l\leq X}e(h\varphi_l(x_l))\bigr|^{2L}\\
  &  B_l=\dsum_{h \in H_1}\bigl|\dsum_{1\leq x_l\leq X}e(h\varphi_l(x_l))\bigr|^{k_1(k_1+1)},
    \end{aligned}
\end{equation*}
and put $m_1=2L-m.$ Then, it follows from H\"older's inequality that
\begin{equation}\label{64}
    \begin{aligned}
    &\dsum_{h\in H_1}\bigl|\dsum_{1\leq \boldsymbol{x}\leq X}e(h(\varphi_1(x_1)+\cdots+\varphi_{s}(x_{s})))\bigr|\\
    &\ll \bigl(\dsum_{h\in H_1}1\bigr)^{1-\left(\frac{m}{2L}+\frac{m_1}{k_1(k_1+1)}\right)}\biggl(\dprod_{l=1}^{m}A_l^{1/2L}\biggr)\biggl(\dprod_{l=m+1}^{m+m_1}B_l^{1/(k_1(k_1+1))}\biggr)X^{s-(m+m_1)}.
    \end{aligned}
\end{equation} 

On recalling the definitions of  $H_1,$ $\mathcal{D}_1$ and $\mathcal{D}_2$, notice that $H_1\subseteq \mathcal{D}_1(\alpha_{1l})$ for $1\leq l\leq m$, and $H_1\subseteq \mathcal{D}_2(\alpha_{1l})$ for $m+1\leq l\leq m+m_1.$ Thus, we have for $1\leq l\leq m$ the bound
$$A_l\leq \dsum_{h\in \mathcal{D}_1(\alpha_{1l})}\bigl|\dsum_{1\leq x_l\leq X}e(h\varphi_l(x_l))\bigr|^{2L}, $$
and for $m+1\leq l\leq m+m_1$ the bound
$$ B_l\leq \dsum_{h \in \mathcal{D}_2(\alpha_{1l})}\bigl|\dsum_{1\leq x_l\leq X}e(h\varphi_l(x_l))\bigr|^{k_1(k_1+1)}.$$
Then, on substituting these inequalities into ($\ref{64}$), it follows by ($\ref{60}$), ($\ref{61}$) and $|H_1|\leq H\ll X^{1-\nu}$ that 
\begin{equation}\label{65}
\begin{aligned}
   & \dsum_{h\in H_1}\bigl|\dsum_{1\leq \boldsymbol{x}\leq X}e(h(\varphi_1(x_1)+\cdots+\varphi_{s}(x_{s})))\bigr| \\
   & \ll X^{1-(\frac{m}{2L}+\frac{m_1}{k_1(k_1+1)})}X^{m}X^{m_1-\frac{m_1\sigma}{k_1(k_1+1)}}X^{s-(m+m_1)-\eta}=X^{\phi-\eta},  
\end{aligned}
\end{equation}
where $\eta$ is suitably small positive number in terms of $\nu$, and 
$$\phi=1-\left(\frac{m}{2L}+\frac{m_1}{k_1(k_1+1)}\right)-\frac{m_1\sigma}{k_1(k_1+1)}+s.$$
Since $m_1=2L-m$ with $m,m_1\geq 0,$
$$\phi=1-\frac{m}{2L}-\frac{(2L-m)(1+\sigma)}{k_1(k_1+1)}+s.$$
On noting $2L\geq k_1^2+(1-2\sigma)k_1+2\sigma,$ simple calculations lead to the lower bound $2L(1+\sigma)\geq k_1(k_1+1).$ Hence, since $\phi$ is a linear function in $m$ with positive slope, we find that the function $\phi$ attains the maximum when $m=2L$, and thus $\phi\leq s.$

%Note that $\phi\leq s$ is equivalent to
%\begin{equation*}
 %   1-\left(\frac{m}{2L}+\frac{m_1}{k_1(k_1+1)}\right)-\frac{\sigma m_1}{k_1(k_1+1)}\leq 0
%\end{equation*}
%Since $m_1=2L-m$, it is equivalent to
%\begin{equation}\label{67}
%    \frac{m}{2L}+\frac{(2L-m)(1+\sigma )}{k_1(k_1+1)}\geq 1.
%\end{equation}
%Since $m$ is between $0$ to $2L$, if we regard the left hand side above as a linear function in $m$, it suffices to show ($\ref{67}$) with $m=0$ and $m=2L$. When $m=2L$, ($\ref{67}$) is trivial. When $m=0$, ($\ref{67}$) is equivalent to
%\begin{equation*}
 %   2L(1+\sigma)\geq k_1(k_1+1).
%\end{equation*}
%Since $2L> k_1^2+(1-2\sigma)k_1+2\sigma$, it suffices to show
%$k_1^2-(1+2\sigma)k_1+2\sigma+2\geq 0$. Since $\sigma\leq 1$, it is derived from 
%\begin{equation*}
 %   \left(k_1+\frac{1+2\sigma}{2}\right)^2+\frac{4\sigma-4\sigma^2+7}{4}\geq 0.
%\end{equation*}
Thus, in all cases, we have
$$\dsum_{h\in H_1}|\dsum_{1\leq \boldsymbol{x}\leq X}e(h(\varphi_1(x_1)+\cdots+\varphi_s(x_{s})))|\ll X^{s-\eta}.$$
 Then, by the same treatment, it follows that for every $H_i\ (i=1,\ldots,2^s)$, we have ($\ref{54}$). This contradicts ($\ref{53}$) stemming from ($\ref{ineq5.35.3}$). Thus, we are forced to conclude that whenever $s>k_1^2+k_1+2\lceil\sigma(1-k_1)\rceil$, one has
\begin{equation*}
    \min_{\substack{0\leq \boldsymbol{x}\leq X\\ \boldsymbol{x}\neq \boldsymbol{0}}}\|\varphi_1(x_1)+\varphi_2(x_2)+\cdots +\varphi_{s}(x_{s})\|\leq H^{-1}.
\end{equation*} 
Hence, by letting $\nu\rightarrow 0$, we complete the proof of Theorem 1.2.
\end{proof}


\begin{thebibliography}{70}
\bibitem{RefJ}\label{ref1}
R. C. Baker, Small solutions of congruences, Mathematika 30 (1983), 164-188.
\bibitem{RefJ1}\label{ref2}
R. C. Baker, Diophantine Inequalities, London Mathematical Society Monographs, New Series, vol. 1, Oxford University Press, Oxford, 1986.
\bibitem{RefJ2}\label{ref3}
R. C. Baker, Small solutions of congruences, II, Funct. Approx. Comment.
Math. 28 (2001), 19–34.
\bibitem{RefJ3}\label{ref4}
R. C. Baker, Small fractional parts of polynomials, Funct. Approx. Comment. Math. 55 (2016), 131–137.
\bibitem{RefJ4}\label{ref5}
E. Bombieri, On Vinogradov's mean value theorem and Weyl sums, in: Automorphic
Forms and Analytic Number Theory, Univ. de Montreal, Monteal, 1990, 7-24.



\bibitem{RefJ5}\label{ref6}
J. Bourgain, C. Demeter, L. Guth, Proof of the main conjecture in Vinogradov’s mean value
theorem for degrees higher than three, Ann. of Math. (2) 184 (2016), no. 2, 633-682.
\bibitem{RefJ6}\label{ref7}
J. Br\"udern, A problem in additive number theory, Math. Proc. Cambridge Philos. Soc. 103 (1988), no. 1, 27–33.
\bibitem{RefJ21}\label{ref21}
I. Danicic, Contributions to number theory, Ph.D. thesis, University of London, 1957.
\bibitem{RefJ19}\label{ref19}
G. Hardy and J. Littlewood, Some problems of Diophantine approximation, Acta Math. 37 (1914), 155-191.
\bibitem{RefJ7}\label{ref8}
D. R. Heath-Brown, Weyl’s inequality, Hua’s inequality, and Waring’s problem, J. Lond. Math. Soc. (2) 38 (1988), no. 2, 216–23
\bibitem{RefJ20}\label{ref20}
H. Heilbronn, On the distribution of the sequence $\alpha n^2$ (mod $1$), Q. J. Math. Oxford Ser. 19 (1948), 249-256.
\bibitem{RefJ24}\label{ref24}
R. A. Hunt, On the convergence of Fourier series, Orthogonal Expansions and
Their Continuous Analogues (Proc.Conf., Edwardsville, 111., 1967), Southern Illinois Univ.
Press, Carbondale, 111., 1968, pp. 235-255. 

\bibitem{RefJ8}\label{ref9}
S. T. Parsell, On the Bombieri-Korobov estimate for Weyl sums, Acta Arith. 138 (2009), 363–372
\bibitem{RefJ9}\label{ref10}
R. C. Vaughan, The Hardy-Littlewood method, 2nd edition, Cambridge University Press, Cambridge, 1997.
\bibitem{RefJ18}\label{ref18}
I.M. Vinogradov, The method of trigonometric sums in the theory of numbers, Trav. Inst. Steklov 23 (1954).
\bibitem{RefJ17}\label{ref17}
H. Weyl, \"Uber die Gleichverteilung von Zahlen mod. Eins, Math. Ann. 77 (1916),  313–352.
\bibitem{RefJ30}\label{ref30}
T. D. Wooley, Breaking classical convexity in Waring's problem: sums of cubes and quasi-diagonal behaviour, Invent. Math. 122 (1995), no. 3, 421-451.
\bibitem{RefJ10}\label{ref11}
T. D. Wooley, New estimates for smooth Weyl sums, J. Lond. Math. Soc.
(2)51 (1995), 1–13.
\bibitem{RefJ11}\label{ref12}
T. D. Wooley, The asymptotic formula in Waring’s problem, Int. Math. Res. Not. IMRN (2012), no. 7, 1485–1504.
\bibitem{RefJ12}\label{ref13}
T. D. Wooley, Mean value estimates for odd cubic Weyl sums, Bull. Lond. Math. Soc. 47 (2015), no. 6, 946-957.
\bibitem{RefJ13}\label{ref14}
T. D. Wooley, The cubic case of the main conjecture in Vinogradov's mean value theorem, Adv. Math. 294 (2016), 532-561.
\bibitem{RefJ14}\label{ref15}
T. D. Wooley, Nested efficient congruencing and relatives of Vinogradov’s mean value theorem,
Proc. Lond. Math. Soc. (3) 118 (2019), no. 4, 942-1016.
\bibitem{RefJ28}\label{ref28}
T. D. Wooley, Rational solutions of pairs of diagonal equations, one cubic and one quadratic, Proc. Lond. Math. Soc. (3) 110 (2015), no. 2, 325-356.
\bibitem{RefJ29}\label{ref29}
A. Zaharescu, Small values of $n^2\alpha$ (mod 1), Invent. Math. 121 (1995), 379–388.
\end{thebibliography}
\end{document}